\documentclass[11pt]{amsart}
\usepackage{amscd,amssymb,palatino}
\usepackage[utf8]{inputenc}
\usepackage[english]{babel}
\usepackage{caption}
\usepackage[leqno]{amsmath}
\usepackage{amssymb}
\usepackage{subcaption}
\usepackage{color}
\usepackage{shadow}
\usepackage{epsfig}
\usepackage{epic}
\usepackage{eepic}
\usepackage{graphics}
\usepackage{graphicx}
\usepackage{psfrag}
\usepackage{hyperref}
\usepackage{array}
\usepackage{calc}
\usepackage{enumitem}
\usepackage[table]{xcolor}

\usepackage{tikz}
\usepackage[dvipsnames,prologue,table]{pstricks}
\usetikzlibrary{arrows,arrows.meta,decorations,patterns,positioning,automata,shadows,fit,shapes,calc,decorations.markings,backgrounds,scopes,decorations.text}

\usepackage{tipa}

\usepackage{fancyhdr}
\fancypagestyle{plain}{\fancyhf{}
\fancyfoot[C]{{\large\sf\bfseries\thepage}}}
\DeclareMathOperator{\Id}{Id}

\usepackage{calc}
\numberwithin{equation}{section}

\newcommand\1{\lower 9pt\hbox{\underbar{}}}
\numberwithin{equation}{section}
\newtheorem*{theorem*}{Main Theorem}
\newtheorem {theorem}{Theorem}[section]
\newtheorem {lemma}[theorem]           {Lemma}
\newtheorem {proposition}[theorem]     {Proposition}

\newtheorem {corollary}[theorem]       {Corollary}

\theoremstyle{definition}
\newtheorem {definition}[theorem]{Definition}
\newtheorem {remark}[theorem]          {Remark}

\newtheorem {example}[theorem]         {Example}

\newcommand{\pr} {\smallskip\noindent{\bf Proof\,\,}}

\newcommand{\R}{\mathbb{R}}

\begin{document}

\title{Contact structures with singularities: from local to global}
\author{Eva Miranda}\address{Eva Miranda,
Laboratory of Geometry and Dynamical Systems, Department of Mathematics, EPSEB, Universitat Polit\`{e}cnica de Catalunya-IMTech
in Barcelona and
\\ CRM Centre de Recerca Matem\`{a}tica, Campus de Bellaterra
Edifici C, 08193 Bellaterra, Barcelona}
\thanks{Eva Miranda and Cédric Oms are partially supported  by the the AEI grant PID2019-103849GB-I00 of MCIN/ AEI /10.13039/501100011033.  Eva Miranda is supported by the Catalan Institution for Research and Advanced Studies via an ICREA Academia Prize 2016 and by the Spanish State
Research Agency, through the Severo Ochoa and Mar\'{\i}a de Maeztu Program for Centers and Units
of Excellence in R\&D (project CEX2020-001084-M). Eva Miranda was supported by a \emph{Chaire d'Excellence} of the \emph{Fondation Sciences Mathématiques de Paris} when this project started and this work
has been supported by a public grant overseen by the French National Research Agency (ANR) as part of the \emph{\lq\lq Investissements d'Avenir"} program (reference:
ANR-10-LABX-0098). This material is based upon work supported by the National Science Foundation under Grant No. DMS-1440140 while the authors were in residence at the Mathematical Sciences Research Institute in Berkeley, California, during the Fall 2018 semester. Cédric Oms is partially supported by the project ANR CoSyDy (ANR-CE40-0014).}
 \email{ eva.miranda@upc.edu}
 \author{ C\'edric Oms}
\address{ C\'edric Oms, BCAM Bilbao Center of Applied Mathematics, Mazarredo Zumarkalea, 14, 48009 Bilbo, Bizkaia}
 \email{coms@bcamath.org}

\subjclass[2000]{53D22,53D25, 37J39, 53D17}

	\begin{abstract}  In this article we introduce and analyze in detail singular contact structures, with an emphasis on $b^m$-contact structures, which are tangent to a given smooth hypersurface $Z$ and satisfy certain transversality conditions. These singular contact structures are determined by the kernel of non-smooth differential forms, called $b^m$-contact forms, having an associated critical hypersurface $Z$. We provide several constructions, prove local normal forms, and study the induced structure on the critical hypersurface. The topology of manifolds endowed with such singular contact forms are related to smooth contact structures via desingularization. The problem of existence of $b^m$-contact structures on a given manifold is also tackled in this paper. We prove that a connected component of a  convex hypersurface of a contact manifold can be realized as a connected component of the critical set of a $b^m$-contact structure. \textcolor{black}{In particular, given an almost contact manifold $M$ with a hypersurface $Z$, this yields the existence of a $b^{2k}$-contact structure on $M$ realizing $Z$ as a critical set}. As a consequence of the desingularization techniques in \cite{gmw1}, we prove the existence of folded contact forms on any almost contact manifold.
	\end{abstract}

	\maketitle
	\section{Introduction}

Contact manifolds have been known for a long time to be the odd-dimensional counterpart to symplectic manifolds. As opposed to symplectic manifolds, contact manifolds are in a way more flexible: Any odd dimensional manifold satisfying an algebraic-topological condition (more precisely almost contact) admits a contact structure  \cite{bem}. The history of contact geometry dates back to Sophus Lie and the study of optics \cite{Geiges}. Contact and symplectic manifolds are closely related and come into the scene as a natural language associated to (Hamiltonian) dynamics. \textcolor{black}{The Hamiltonian dynamics induces an interesting class of vector fields when restricted to special level sets of the Hamiltonian, , namely Reeb vector fields. A particular but  outstanding source of examples of Reeb vector fields is provided by geodesic flows on a Riemannian manifold.}


In this alluring connection between symplectic and contact manifolds a natural aspect has been neglected: Can we consider \emph{singular forms}?
This is too wild as a question as singularities can be too complicated. However in the last years, a class of singular symplectic forms called $b^m$-symplectic forms has been widely explored by several authors \cite{GMP, gmps, km}. In this article, we do a first step towards studying the local geometry and topology of the odd-dimensional counterpart of those manifolds, called $b^m$-contact manifolds. \textcolor{black}{These structures appear in a collection of problems ranging from  celestial mechanics (as  the so-called restricted planar circular three body problem, see \cite{MO2}) to fluid mechanics (see \cite{MOP}). Geodesic flows on manifolds with boundary and on pseudo-Riemannian manifolds also yield rich sources of examples.}

	Those manifolds constitute a generalization of contact manifolds, where the non-integrability condition is satisfied on a dense subset, but where the distribution is integrable on a hypersurface. Alternatively, this can be thought of as studying contact manifolds with boundary, where the distribution is tangent to the boundary. 

	Those distributions are described as the kernel differential $1$-forms satisfying the usual contact condition away from the hypersurface, but which are singular on that given hypersurface. The language of those non-smooth forms in the case of manifolds with boundary is not new. The notion of $b$-tangent bundle was already introduced by Melrose in \cite{Me} as a framework to study differential calculus on manifolds with boundary. Recently, it regained a lot of attention in the Poisson and symplectic setting. Indeed, in the foundational work of Radko \cite{R}, a certain type of Poisson structures on closed surfaces (\emph{topologically stable Poisson surfaces}) are classified. Later, in \cite{GMP}, it was shown that those Poisson structures can be studied under the eyes of symplectic geometry associating a symplectic form over the $b$-cotangent bundle. Since then many efforts have been united to understand the local and global behavior of this generalization of symplectic manifolds, see for example \cite{BDMOP,FMM,GL} and references therein. This article can be regarded to be the first one to consider an odd-dimensional counterpart of the story.
	
	The investigation of existence of contact structures on any odd-dimensional manifold has a particularly rich history and led to many important developments in the field.  In this article we consider the singular analogue and  provide an answer in our setting by narrowly linking the existence problem of singular contact structures to convex hypersurfaces in contact geometry, thereby shedding new light on the theory of convex surfaces initiated by \cite{Giroux}. 	We also relate the topology of $b^m$-contact manifolds, depending on the parity of $m$, to the one of contact manifolds, or to the  so-called \emph{folded contact} manifolds.

	\textbf{Organization of this article:}
	
We start by \textcolor{black}{outlining some examples that motivate the study of $b$-contact structures.} We then review the basics of $b$-symplectic geometry in Section \ref{b-symp} by  explaining in greater details the construction of the $b$-tangent bundle and the extension of the de Rham exterior derivative. We also include a selection of results in $b$-symplectic geometry that are used in this article.
We then give the main definition of this article, namely the one of $b$-contact manifolds. We prove local normal forms for $b$-contact forms in Section \ref{Darboux}. We will see in Section \ref{bJacobi} that the right framework to study those geometric structures is the one of Jacobi manifolds. The induced structure by the $b$-contact structure on the boundary is explained in Section \ref{structurecriticalset}.
	We continue by explaining the relation with $b$-symplectic geometry in Section \ref{symplectization}. {In Section \ref{sec:othersingularities} we introduce contact structures admitting more general singularities, as for example $b^m$-contact structures, but also structures dual to $b$-contact structures, that we call \emph{folded contact} structures.} The existence of singular contact structures on a prescribed manifold is dealt with in Sections \ref{sec:desingularization} and \ref{sec:existence}: Namely, we explore the relation of $b^m$-contact manifolds to smooth contact structures following the techniques of \cite{gmw1} and proving existence theorems for $b^m$-contact structures on a given manifold. The constructions in Section \ref{sec:desingularization} and \ref{sec:existence} rely strongly on the existence of convex hypersurfaces on contact manifolds but also on the desingularization constructions in \cite{gmw1} and on new \emph{singularization} techniques. {We end the article with a discussion on open problems concerning the dynamical properties of the Reeb vector field associated to $b^m$-contact forms and possible applications to celestial mechanics.}	An appendix on the local analysis of Jacobi manifolds is included.
	
	\textbf{Acknowledgements:} We are grateful to Boris Khesin and  Charles-Michel Marle for several key conversations  during the preparation of this article. We are indebted to the Fondation Sciences Mathématiques de Paris for endowing the first author with a Chaire d'Excellence in 2017-2018 when this adventure started and to the Observatoire de Paris for being a source of inspiration for many of the constructions in this article and for their hospitality during the stay of both authors during the Fall and Winter of 2017-2018.

\section{Motivating Examples}\label{sec:motivating}

\textcolor{black}{As mentioned in the introduction, smooth contact manifolds appear for instance as regular level-sets for a Hamiltonian function $H$ in symplectic manifolds, whenever there is a Liouville vector field transverse to the level-set. These level-sets are called contact-type hypersurfaces. This is especially interesting as the associated Reeb vector field is a reparametrization of the Hamiltonian vector field. Thus understanding the dynamical behavior of the Hamiltonian vector field on the level-set boils down to understanding the dynamics of the Reeb vector field. As was shown in \cite{MO2}, this picture remains true in the set-up of $b$-symplectic manifolds. In this section we list two sources of examples where $b$-contact forms appear, motivating the study of the local structure of those geometric structures with singularities.}

\subsection{The restricted circular planar three body problem}

\textcolor{black}{The restricted circular planar three body problem deals with the modeling of the motion of a small body in the Euclidean $3$-dimensional space, exposed to the gravitational potential of two massive bodies that move in circles around their center of gravity. The motion of the small body  is constraint to the plane spanned by the motion of the two heavy bodies \textcolor{black}{(called primaries in the literature)}. In the last decade, \textcolor{black}{the existence of contact-type hypersurfaces has been proved in this set-up}, see \cite{contactmoon}. This has opened the door for applications of deep results coming from contact topology to this classical problem.}

More precisely, the Hamiltonian of the restricted circular planar three body problem is given by $$ H(q,p,t)= \frac{|p|^2}{2} - U(q,t), (q,p) \in \mathbb R^2\setminus \{q_1(t),q_2(t)\} \times \mathbb R^2, $$ where $U(q,t)$ denotes the gravitational potential $U(q,t)= \frac{1-\mu}{|q-q_1(t)|} + \frac{\mu}{|q-q_2(t)|}$. Here $\mu_1,\mu_2\in (0,1)$ are the relative masses of the two massive bodies \textcolor{black}{ and $q_1(t)$ and $q_2(t)$ stand for the positions of the primaries}.

The McGehee transformation is classically used when studying the manifold at infinity. To describe this change of coordinates, first pass to polar coordinates $(q, p) \mapsto (r,\alpha, P_r, P_\alpha)$, where $q=(r \cos \alpha,r\sin \alpha)$ and change the momenta accordingly $p=(P_r\cos \alpha-\frac{P_\alpha}{r}\sin \alpha,P_r\sin \alpha+\frac{P_\alpha}{r}\cos \alpha)$ to obtain a symplectic change of coordinates. In the coordinates $(r,\alpha,P_r,P_\alpha)$, the McGehee change of coordinates is given by
$$r=\frac{2}{x^2}, x \in \mathbb R^{+}.$$

This is a non-canonical transformation of the configuration space and therefore the canonical symplectic structure of the cotangent bundle is not preserved. However the induced structure is a $b^3$-symplectic structure given by
$$\omega=-\frac{4}{x^3}dx\wedge dP_r+d\alpha \wedge dP_\alpha.$$ As  proved in Theorem 4.4. in \cite{MO2}, positive level-sets of the Hamiltonian function $H$ admit a transverse Liouville vector field and therefore admit an induced $b^3$-contact structure.

\subsection{\textcolor{black}{Geodesic flows revisited: The $b$-geodesic flow and geodesic flows on pseudo-Riemannian manifolds}}

 \textcolor{black}{\textcolor{black}{ The geodesic flow  of  a Riemannian manifold $(M,g)$} is a special case of Hamiltonian flow. As  shown for instance in \cite{Geiges}, Theorem 1.5.2, the geodesic flow can be seen as the Hamiltonian flow  restricted to the level-set $H^{-1}(2)$ on the cotangent bundle $T^*M$ with Hamiltonian given by $H(q,p)=\frac{1}{2}||{p}||^2$, where $q$ are the position variables and $p$ the associated momenta  (the norm here is induced by the dual metric). This is thus equivalent to the Reeb flow on the unit cotangent bundle $S^*M=H^{-1}(2)$, as the vector field that is radial in the momenta is transverse to the level-set $H^{-1}(2)$.}

Whenever the manifold exhibits asymptotically cylindrical ends, it is often convenient to consider its compactification, as is done in Chapter 2 in \cite{Me}. The resulting manifold is a manifold with boundary and the Riemannian metric ceases to be smooth on the boundary, but one obtains what is called a \emph{b-metric}. This is a bundle metric on the $b$-tangent bundle, which will be precisely defined in the next section. The geodesic flow of this $b$-metric can be seen to be the Reeb flow of the $b$-contact manifold $S^*M$. This is explained in further detail in Example \ref{ex:bgeod}.

\textcolor{black}{This rich connection between geometry and bundles can be taken to the pseudo-Riemannian realm. The following results are contained in \cite{KT}.}

\textcolor{black}{Let $M$ be a smooth manifold with a pseudo-Riemannian metric $g$. The geodesic flow in $T^*M$ is the Hamiltonian vector field $X_H$ of $H(q,p)=g(p,p)/2$.
The set of all oriented geodesic lines $\mathcal{L}$ can be decomposed as  ${\mathcal{L}}={\mathcal{L}}_+\cup   {\mathcal{L}}_-\cup  {\mathcal{L}}_0$  with ${\mathcal{L}}_+, {\mathcal{L}}_-, {\mathcal{L}}_0$  the space of oriented non-parameterized space-, time- and light-like geodesics (that is, $H=const>0, <0$ or $=0$, respectively).  The set  ${\mathcal{L}}_0$ is the common boundary of ${\mathcal{L}}_\pm$.}

\textcolor{black}{We recall from \cite{KT} the following:}

\begin{theorem}[Khesin-Tabachnikov]
The manifolds ${\mathcal{L}}_\pm$ carry symplectic structures which are described  from $T^*M$ by Hamiltonian reduction on the level hypersurfaces $H=\pm 1$. The manifold ${\mathcal{L}}_0$ carries a contact structure whose symplectization is the Hamiltonian reduction of the symplectic structure in $T^*M$ (without the zero section)
on the level hypersurface $H=0$.
\end{theorem}

\textcolor{black}{In \cite{KT} the authors prove that in general the symplectic structures on ${\mathcal{L}}_\pm$  do not extend to a Poisson structure. However in the $2$-dimensional case there is a Poisson structure extending it. This structure can be described with the $b$-glasses as a $b$-symplectic form.}

\textcolor{black}{In the case of the Lorentzian plane  with the (pseudo-Riemannian) metric
$ds^2=dxdy$ .
The light-like lines  can be identified as the horizontal and vertical lines. The space-like lines are lines with positive slope and  the time-like  ones have negative slope. The spaces ${\mathcal{L}}_+$ and ${\mathcal{L}}_-$ have two components each.
Let us compute the symplectic form on ${\mathcal{L}}_+$. Take coordinates on the first coordinate quadrant
such that the unit directing vector of a line reads $(e^{-u},e^u),\ u\in \mathbb R$.  Write as $r(e^{-u},-e^u),\ r\in \R$, the perpendicular to the line from the origin. Then $(u,r)$ are coordinates in ${\mathcal{L}}_+$. Symmetrically,  it can be done for ${\mathcal{L}}_-$.}

\textcolor{black}{As proved in \cite{birman}, the symplectic structures can be computed by the following lemma.}

\begin{lemma}  \label{boris}
 The symplectic form $\omega$ on  ${\mathcal{L}}_+$ can be written as  $2 du\wedge dr$, and as $-2 du\wedge dr$ on ${\mathcal{L}}_-$.
 \end{lemma}
This symplectic structure extends to the (singular) $2$-form on $\mathcal{L}$ which is described in Remark 2.8 in \cite{KT}.
 In the coordinates mentioned above the extension form can be explicitly computed. We follow \cite{KT} for the computations.  Use now the standard notation for a line $(1,\epsilon)$. By imposing the equality of the  slope of $(1,\epsilon)$ to the slope of the line $(e^{-u},e^u)$ expressed in the above coordinates the following relation $u=\frac{1}{2}\ln\epsilon$ is obtained. Now let us make the appropriate substitutions in the symplectic form computed in Lemma \ref{boris} and the form
$\omega =2du\wedge dr$  becomes $d\ln\epsilon\wedge dr=\frac{1}{\epsilon}\, d\epsilon\wedge dr$. Thus the form $\omega\to\infty$ as $\epsilon\to 0 $ and the symplectic structure on each of the spaces ${\mathcal{L}}_+$  and ${\mathcal{L}}_-$  goes to infinity
as one approaches the one-dimensional manifold ${\mathcal{L}}_0$.  This structure blows up when it gets close to ${\mathcal{L}}_0$ but its inverse  minutely extends the symplectic structures
across  $\mathcal{L}_0$ to a smooth Poisson structure.

 \textcolor{black}{Both examples, the unit $b$-cotangent bundle and the $b$-symplectic structures, show that geodesic flows are a source of singular examples. We will go back to this pseudo-Riemmannian structures in higher dimensions. In higher dimensions the set $\mathcal{L}$ is not endowed with a Poisson structure but it can be seen as the critical set of a singular contact structure ($b$-contact structures).}

\section{Preliminaries: $b$-Symplectic survival kit}\label{b-symp}
	
	Let $(M^n,Z)$ be a smooth manifold of dimension $n$ with a hypersurface $Z$. In what follows, the hypersurface $Z$ will be called \emph{critical hypersurface} and the couple $(M,Z)$ $b$-manifold. Assume that there exists a global defining function for $Z$, that is $f:M\to \R$ such that $Z=f^{-1}(0)$, $0$ being a regular value of $f$. A vector field is said to be a $b$-vector field\footnote{The letter $b$ stands for \emph{boundary} as introduced by \cite{Me}.} if it is everywhere tangent to the hypersurface $Z$. The space of $b$-vector fields is a Lie sub-algebra of the Lie algebra of vector fields on $M$. A natural question to ask is whether or not there exists a vector bundle such that its sections are given by the $b$-vector fields. A coordinate chart of a neighbourhood around a point $p \in Z$ is given by $\{(x_1,\dots,x_{n-1},f)\}$ and the $b$-vector fields restricted to this neighbourhood form a locally free $C^\infty$-module with basis
	$$(\frac{\partial}{\partial x_1},\dots,\frac{\partial}{\partial x_{n-1}}, f\frac{\partial}{\partial f}).$$
	
	By the Serre--Swan theorem \cite{S}, there exists an $n$-dimensional vector bundle whose sections are given by the $b$-vector fields. We denote this vector bundle by ${^b}TM$, the \emph{$b$-tangent bundle}. \textcolor{black}{Alternatively, the $b$-tangent bundle at $p\in Z$ can be defined, following \cite{Me} (Lemma 2.5) as the set of $b$-vector fields modulo the product of  $b$-vector fields with the ideal of functions vanishing at $p$. This shows that the definition is coordinate independent.}
	
	\textcolor{black}{Given a $b$-vector field $v$, the restriction of $v$ to $Z$ is every tangent to the hypersurface. The restriction of $b$-vector fields to the critical set given by $\Gamma({^b}TM) \to \Gamma(TZ)$ is a morphism of $C^\infty(Z)$-modules and is thus induced by a vector bundle morphism
	\begin{equation}\label{eq:restrictionbbundle}
	{^b}TM|\textcolor{black}{_Z} \to TZ.
	\end{equation}
As shown in \cite{GMP}, Proposition 4, the kernel of Equation (\ref{eq:restrictionbbundle}) is a line-bundle, denoted by $\mathbb{L}_Z$, with a canonical non-vanishing section. In the above coordinates, this canonical non-vanishing section is spanned around the critical hypersurface by the vector field $f\frac{\partial}{\partial f}$. This vector field is called the \emph{normal $b$-vector field}. } \textcolor{black}{It is shown in \cite{GMP}, Proposition 4, that the splitting of Equation \ref{eq:btangentsplitting} does not depend on the choice of the defining function $f$ of the critical set.}

{We hence obtain the splitting at $p\in Z$ given by}
\begin{equation}\label{eq:btangentsplitting}
{^b}T_pM=T_pZ+\mathbb{L}_p
\end{equation}

\begin{remark}\label{rem:vanishing}
\textcolor{black}{For the sake of clarity, we emphasize the following. Note that the normal $b$-vector field does not vanish as a section of ${^b}TM$. However, the $b$-vector field can be viewed as a smooth vector field by the inclusion $\Gamma({^b}TM) \to \Gamma(TM)$. Thus when viewed as a section of the tangent bundle, it does vanish. \textcolor{black}{This will be of importance in Definition \ref{def:singularReeb}}.}
\end{remark}

	We now adopt the classical construction to obtain differential forms for this vector bundle. We denote the dual of this vector bundle by ${^b}T^*M:=({^b}TM)^*$ and call it the \emph{$b$-cotangent bundle}. \textcolor{black}{Away from the critical hypersurface, the $b$-cotangent bundle agrees with the cotangent bundle, ${^b}T^*M=T^*M$, and on the critical set, there is an embedding $T_p^*Z \to {^b}T_p^*M$, $p\in Z$, whose image is given by $\textcolor{black}{E_p}:=\{l\in {^b}T_p^*M | l(w_p)=0\}$, where $w_p$ is the normal $b$-vector field. Given a defining function $f$, the differential form $\frac{df}{f}$ is not well-defined at the critical hypersurface, but the evaluation $\frac{df}{f}(v)$, where $v\in \Gamma({^b}TM)$, extends smoothly to $Z$ and therefore $\frac{df}{f}$ yields a section of ${^b}T^*M$. At the critical hypersurface, the splitting
$${^b}T_p^*M=\big(\frac{df}{f}\big)_p+\textcolor{black}{E_p}$$ holds.} \textcolor{black}{Note that the splitting depends on the choice of the defining function $f$: let $\tilde{f}$ be a different defining function for the critical set $Z$, so $\tilde{f}=g{f}$ for some non-vanishing function $g$. Then $\dfrac{d\tilde{f}}{\tilde{f}}=\dfrac{df}{f}+dg$.}    
	
	A $b$-form of degree $k$ is a section of the $k$th exterior wedge product of the $b$-cotangent bundle: $\omega \in \Gamma\big(\Lambda^k({^b}T^*M)\big):= {^b}\Omega^k(M)$. To extend the de Rham differential to an exterior derivative for $b$-forms, we need a decomposition lemma.
	
	\begin{lemma}[Proposition 5 in \cite{GMP}]\label{decomposition}
		Let $\omega \in {^b}\Omega^k(M)$ be a $b$-form of degree $k$. Then $\omega$ decomposes as follows:
		$$ \omega = \frac{df}{f}\wedge \alpha +\beta, \quad \alpha \in \Omega^{k-1}(M),\ \beta \in \Omega^k(M)$$
		\textcolor{black}{and at $p\in Z$, $\alpha_p$ and $\beta_p$ are unique.}
	\end{lemma}
\textcolor{black}{Note that $\beta_p$ depends on the choice of the defining function $f$, as the expression $\frac{df}{f}$ also depends on this choice.}
	Equipped with this decomposition lemma, we extend the exterior derivative by putting
	$$d\omega:= \frac{df}{f}\wedge d\alpha+d\beta \quad \text{ for } \omega \in {^b}\Omega^k(M).$$
	
	\textcolor{black}{The right hand side is well defined and agrees with the De Rham differential on $M\setminus Z$. It also extends smoothly over $M$ as a section of $\Lambda^{k+1}({^b}T^*M)$} and is therefore indeed an extension of the usual exterior derivative and it holds that $d^2=0$. \textcolor{black}{Furthermore, we say that a $b$-form at $p\in Z$ is an \emph{honest} $b$-form if $\alpha|_p \neq 0$, where $\alpha$ is as in the decomposition lemma (Lemma \ref{decomposition}). This means that if $\omega_p$ is not an honest $b$-form, then $\omega_p$ is a smooth de Rham form at $p$.}
	
	\begin{definition}
		An even-dimensional $b$-manifold $M^{2n}$ with a $b$-form $\omega \in {^b}\Omega^2(M)$ is $b$-symplectic if $d\omega =0$ and $\omega^n \neq 0$ as element of $\Lambda^{2n}({^b}T^*M)$.
	\end{definition}
	
 {Taking into account the decomposition as in Lemma \ref{decomposition}, the condition $\omega^n \neq 0$ can be written as $\frac{df}{f}\wedge \alpha \wedge \beta^{n-1}\neq 0$.}
	
	Outside of the critical set $Z$, we are dealing with symplectic manifolds. On the critical set, the local normal form of the $b$-symplectic form is given by the following theorem.
	
	\begin{theorem}[$b$-Darboux theorem, Theorem 37 in \cite{GMP}]\label{thm:symplecticbdarboux}
		Let $\omega$ be a $b$-symplectic form on $(M^{2n},Z)$. Let $p\in Z$. Then we can find a local coordinate chart $(x_1,y_1,\ldots,x_n,y_n)$ centered at $p$ such that hypersurface $Z$ is locally defined by $y_1=0$ and
		$${\omega=dx_1\wedge\frac{d y_1}{y_1}+\sum_{i=2}^n dx_i\wedge dy_i.}$$
	\end{theorem}
	
	The $b$-Darboux theorem for $b$-symplectic forms has been proved using two different approaches. The first proof follows Moser path method, that can be adapted in the $b$-setting. Another way of proving it is to show that a $b$-form of degree $2$ on a $2n$-dimensional $b$-manifold is $b$-symplectic if and only if its dual bi-vector field is a Poisson vector field $\Pi$ whose maximal wedge product is transverse to the zero section of the vector bundle $\Lambda^{2n}({^b}TM)$, that is $\Pi^n \pitchfork 0$. A Poisson manifold satisfying this condition is called a \emph{$b$-Poisson manifold}. Using the transversality condition in Weinstein's splitting theorem (see \cite{Weinstein}), one sees that the Poisson structure is of the form
	\begin{equation}
	\Pi = y_1\frac{\partial}{\partial x_1}\wedge
	\frac{\partial}{\partial y_1}+\sum_{i=2}^n \frac{\partial}{\partial x_i}\wedge
	\frac{\partial}{\partial y_i}.
	\end{equation}
	
	Furthermore, Weinstein splitting theorem implies that the critical set of a $b$-symplectic manifold is a regular codimension one foliation of symplectic leaves. Even better, it is proved in \cite{GMP} that the critical set is a cosymplectic manifold\footnote{A cosymplectic manifold is manifold $M^{2n+1}$ together with a closed one-form $\eta$ and a closed two-form $\omega$ such that $\eta\wedge\omega^n$ is a volume form.}.
	
 {Higher order singularities were considered in \cite{Scott}, called $b^m$-symplectic forms (see Section \ref{sec:othersingularities}).} The relation of $b^m$-symplectic manifolds to symplectic manifolds and the less well-known folded  symplectic manifolds was investigated in \cite{gmw1}.
	\begin{theorem}[Theorems 3.1 and  5.1 in \cite{gmw1}]\label{symplecticdesingularization}
		{ Let $\omega$ be a} $b^m$-symplectic structure  on a manifold $M$  and let $Z$ be its critical hypersurface.
		\begin{itemize}
			\item If $m=2k$, there exists  a family of symplectic forms ${\omega_{\epsilon}}$ which coincide with  the $b^{m}$-symplectic form
			$\omega$ outside an $\epsilon$-neighborhood of $Z$ and for which  the family of bivector fields $(\omega_{\epsilon})^{-1}$ converges in
			the $C^{2k-1}$-topology to the Poisson structure $\omega^{-1}$ as $\epsilon\to 0$ .
			\item If $m=2k+1$, there exists  a family of folded symplectic forms ${\omega_{\epsilon}}$ which coincide with  the $b^{2k+1}$-symplectic form
			$\omega$ outside an $\epsilon$-neighborhood of $Z$.
		\end{itemize}
	\end{theorem}
	We say that $(M,\omega_\epsilon)$ is the $f_\epsilon$-desingularization of $(M,\omega)$.
	A direct consequence of this theorem is that any orientable manifold admitting a $b^{2k}$-symplectic structure admits a symplectic structure.

	\section{$b$-Contact manifolds}\label{b-contact}
	In this section we introduce the main objects of this article. Inspired by the definition of $b$-symplectic manifolds, we define the contact case as follows:
	
	\begin{definition}
		Let $(M,Z)$ be a (2n+1)-dimensional $b$-manifold. A $b$-contact structure is the distribution given by the kernel of a $b$-form of degree $1$, $\xi=\ker \alpha \subset {^b}TM$, $\alpha \in {^b\Omega^1(M)}$, that satisfies $\alpha \wedge (d\alpha)^n \neq 0$ as a section of $\Lambda^{2n+1}(^bT^*M)$. We say that $\alpha$ is a $b$-contact form and the pair $(M,\xi)$ a $b$-contact manifold.
	\end{definition}
	
	The hypersurface $Z$ is called \emph{critical hypersurface}. In what follows, we always assume that $Z$ is non-empty. Away from the critical set $Z$ the $b$-contact structure is  a smooth contact structure. The former definition fits well with what is standard in contact geometry where coorientable contact manifolds are considered (i.e. there exists a defining contact form with kernel given contact structure).
	
	\begin{example}\label{extendedphasespace}
		Let $(M,Z)$ be a $b$-manifold of dimension $n$. Let $z,y_i, i=2,\dots,n$ be the local coordinates for the manifold $M$ on a neighbourhood of a point in $Z$, with $Z$ defined locally by $z=0$ and $x_i, i=1,\dots,n$ be the fiber coordinates on $^{b}T^*M$, then the Liouville one-form is given in these coordinates by
		$$x_1\frac{dz}{z}+\sum_{i=2}^{n}x_idy_i.$$
		The bundle $\R\times{^{b}T^*M}$ is a $b$-contact manifold with $b$-contact structure defined as the kernel of the one-form
		$$dt+x_1\frac{dz}{z}+\sum_{i=2}^{n}x_idy_i,$$
		where $t$ is the coordinate on $\R$. The critical set is given by $\tilde{Z}=Z\times \R$. Using the definition of the extended de Rham derivative, one checks that $\alpha\wedge (d\alpha)^n \neq 0$. Away from $\tilde{Z}$, $\xi = \ker \alpha $ is a non-integrable hyperplane field distribution, as in usual contact geometry. On the critical set however, $\xi$ is tangent to $\tilde{Z}$. This comes from the definition of $b$-vector fields. As a smooth distribution, that is a viewing $\xi$ as a distribution of $TM$, the rank can drop by $1$ on $\tilde{Z}$, and hence we cannot say that $\xi$ defines a smooth hyperplane field in $TM$.
	\end{example}
	
	As we will see in the next example, the rank does not necessarily drop.
	
	\begin{example}\label{singularReebexample}
		Let us take $\R^{2n+1}$ with coordinates $(z,x_1,\dots,x_n,y_1,\dots,y_n)$. We consider the distribution of the kernel of $\alpha=\frac{dz}{z}+\sum_{i=1}^n x_i dy_i$. The critical set is given by $z=0$ and the rank does not drop on the critical set: on the critical set, the distribution is spanned by $\{\frac{\partial}{\partial x_i},\frac{\partial}{\partial y_i}, i=1,\dots n\}$.
	\end{example}
	
	Using the two last examples and a generalization of Moebius transformations, we can construct $b$-contact structures on the unit ball with critical set given by the unit sphere minus a point.
	
	\begin{example}\label{unitball}
		Let us denote the unit ball of dimension $n$ by $D^n$ and the half-space, that is $\R^n$ where the first coordinate is positive, by $\R^n_+$. The Moebius transformation maps the open half-space diffeomorphically to the open $2$-disk minus a point by the following map:
		\begin{align*}
		\Phi: \{z\in \mathbb{C}| \Re(z)>0\} &\to D^2 \setminus \{(1,0)\} \\
		z &\mapsto \frac{z-1}{z+1}.\\
		\end{align*}
		This map can easily be generalized to all dimensions and the inverse is given by
		\begin{align*}
		\Psi: D^{n} \setminus \{(1,0,\dots,0)\} &\to \R^n_+ \\
		(x_1,\dots,x_{n})&\mapsto \frac{1}{(x_1-1)^2+\sum_{i=2}^{n}x_i^2}\big( 1-\sum_{i=1}^{n}x_i^2,2x_2,\dots,2x_{n}\big).
		\end{align*}
		We now endow $\R^{2n+1}_+$ with the $b$-contact structures described in Example \ref{extendedphasespace} (respectively \ref{singularReebexample}) and pull-back the $b$-contact form. We obtain hence two different $b$-contact structures on the unit ball minus a point and the critical set is given by the unit sphere $S^{2n-2}$ minus the point $(1,0,\dots,0)$.
	\end{example}

	It is not possible to compactify this example by adding the point. This can be seen when computing the hyperplane distribution of the pushforward under $\Phi$. Alternatively, this follows as we will see from Corollary \ref{corHam3dim}. However, we will see that the $3$-sphere does admit a $b$-contact structure, induced by a $b$-symplectic structure, see Example \ref{example:S3}.

	\begin{example}\label{compactexample}
		A compact example admitting a $b$-contact structure is given by $S^2 \times S^1$. Let us consider the $2$-sphere $S^2$, with coordinates $(\theta,h)$ where $\theta\in [0,2\pi]$ is the angle and $h\in [-1,1]$ is the height, and the $1$-sphere $S^1$ with coordinate $\varphi \in [0,2\pi]$. Then $(S^2 \times S^1,\alpha=\sin \varphi d \theta + \cos \varphi \frac{dh}{h})$ is a $b$-contact manifold. Once more, the rank on the critical set changes when $\cos \varphi =0$, where instead of a plane-distribution, we are dealing with a line distribution.
	\end{example}
	
	\begin{example}[Non-orientable example]\label{ex:nonorientable}
	 Contact manifolds with a coorientable contact structure are always orientable as $\alpha\wedge (d\alpha)^n$ is a volume form. There are $b$-contact forms on non-orientable manifolds. Consider the example of the $b$-contact form on the $3$-torus given by $(\mathbb{T}^2\times S^1,\alpha=\cos\theta \frac{dx}{\sin 2\pi x}+\sin \theta dy)$. Consider the group  $\mathbb{Z}/2\mathbb{Z}$ that acts on $(x,y)\in \mathbb{T}^2$ by $\Id \cdot (x, y) = (x, y)$ and $-\Id \cdot (x, y) = (1-x, y)$. Its quotient space is the Klein bottle. The $b$-contact form is invariant under the action of the group and therefore descends to $\mathbb{K}\times S^1$ where $\mathbb{K}$ is the Klein bottle. The manifold $\mathbb{K}\times S^1$ is of course non-orientable.
	\end{example}

	\begin{example}[Product examples] Let $(N^{2n+1}, \alpha)$ be a $b$-contact manifold and let $(M^{2m}, d\lambda)$ be an exact symplectic manifold, then $(N\times M, \alpha+\lambda)$ is a $b$-contact manifold. It is easy to check that $\tilde{\alpha}=\alpha+\lambda$ satisfies $\tilde{\alpha}\wedge (d\tilde{\alpha})^{n+m}\neq 0$.
		
		In the same way  if $(N^{2n+1}, \alpha)$ is a contact manifold and  $(M^{2m}, d\lambda)$ an exact $b$-symplectic manifold (where exactness is understood in the $b$-complex), then $(N\times M, \alpha+\lambda)$ is a $b$-contact manifold.
		These product examples can even be endowed with additional structures such as group actions or integrable systems. For instance we can produce examples of toric $b$-contact manifolds combining  the product of toric contact manifolds in \cite{lerman} with (exact) toric $b$-symplectic manifolds (see \cite{gmps}).
		We can also combine the techniques in \cite{km} for $b$-symplectic manifolds and \cite{Boyer} (among others) for contact manifolds to produce examples of integrable systems on these manifolds.
		
	\end{example}
	
{Besides these examples, more geometric in nature, $b$-contact structures appear when performing regularization techniques in well-known problems in celestial mechanics. More precisely, in \cite{MO2}, positive energy level-sets of the Hamitonian are shown to admit $b^3$-contact structures (see Section \ref{sec:othersingularities} for the definition of $b^3$-contact structures).}
	
	\section{The $b$-contact Darboux theorem} \label{Darboux}
	
	In usual contact geometry, the Reeb vector field $R_\alpha$ of a contact form $\alpha$ is given by the equations
	$$\begin{cases}
	\iota_{R_\alpha}d\alpha=0 \\ \alpha(R_\alpha)=1.
	\end{cases}$$
	In the case where we change the tangent bundle by ${^b}TM$, the existence is given by the same reasoning: $d\alpha$ is a bilinear, skewsymmetric $2$-form on the space of $b$-vector fields $^bTM$, hence the rank is an even number. As $\alpha \wedge (d\alpha)^{n}$ is non-vanishing and of maximum degree, the rank of $d\alpha$ must be $2n$, its kernel is $1$-dimensional and $\alpha$ is non-trivial on that line field. So a global vector field is defined by the normalization condition.
	
	By the same reasoning, we can define the $b$-contact vector fields: for every function $H \in C^\infty(M)$, there exists a unique $b$-vector field $X_H$ defined by the equations
	\begin{equation}\label{eq:Hamiltonianvf}
	\begin{cases}
	\iota_{X_H}\alpha=H \\
	\iota_{X_H}d\alpha=-dH+R_\alpha(H)\alpha.
	\end{cases}
	\end{equation}
	
	A direct computation yields that in Example \ref{extendedphasespace}, the Reeb vector field is given by $\frac{\partial}{\partial t}$. In Example \ref{singularReebexample}, the Reeb vector field is given by $z\frac{\partial}{\partial z}$ and \textcolor{black}{hence it vanishes when viewed as a smooth vector field on $Z$. We emphasize that $z\frac{\partial}{\partial z}$ does not vanish as a section of ${^b}TM$, see also Remark \ref{rem:vanishing}.} \textcolor{black}{In order to state the local classification of $b$-contact forms, we introduce the following definitions.}
	
	\begin{definition}\label{def:singularReeb}
	\textcolor{black}{We say that the Reeb vector field $R_\alpha$ associated to a $b$-contact form $\alpha$ is singular at a point $p\in Z$ if $i(R_\alpha)_p=0$, where $i:\Gamma({^b}TM) \to \Gamma(TM)$.}
		\end{definition}
		
		\textcolor{black}{Similarly, the $b$-contact distribution given by the kernel of a $b$-contact form defines a distribution of codimension $1$ in ${^b}TM$, however the image under the inclusion $i$ of the vector fields that span the kernel of $\alpha$ may vanish.}

\begin{definition}\label{def:singularcontactstructure}
The $b$-contact distribution $\xi=\ker \alpha$ is said to be singular at $p\in Z$ if the rank of the distribution drops by one at $p$ when viewed as a smooth distribution. In other words it is singular if there exists $v \in \xi$ such that $ \Gamma({^b}T_pM) \ni v_p\neq 0$ but \textcolor{black}{$i(v_p)=0$}, where
	\begin{equation}\label{eq:binclusion}
		i:\Gamma({^b}TM) \to \Gamma(TM)
	\end{equation}
 denotes the inclusion. If $\xi_p$ is not singular, it is said to be \emph{regular} at $p$.
\end{definition}

\begin{example}
\textcolor{black}{Consider the $b$-contact form $dx+y\frac{dz}{z}$ on $\mathbb{R}^3$. The $b$-vector field $y\frac{\partial }{\partial x}-z\frac{\partial}{\partial z}$ belongs to the $b$-contact distribution is non-vanishing as a $b$-vector field but zero as a smooth vector field on the critical set intersecting $\{y=0\}$. The $b$-contact distribution $\ker \alpha$ is singular at $(x,0,0)$.}
\end{example}
	
	We now prove a Darboux theorem for $b$-contact forms. We will see that, roughly speaking, the Reeb vector field locally classifies $b$-contact structures.  The proof follows the one of usual contact geometry as in \cite{Geiges}. More precisely, it makes use of Moser's path method. There are two differences from the standard Darboux theorem: the first one is that there exist two local models, depending on whether or not the Reeb vector field is vanishing on the critical set $Z$. The second one is that in the case where the Reeb vector field is singular, the local expression of the contact form only holds \textcolor{black}{at the point $p\in Z$ and not in a local neighbourhood around $p$}, see for instance Example \ref{Exampleonepointsingular}. Furthermore, in this case, \textcolor{black}{the expression of the $b$-contact form at $p$ holds only up to multiplication of a non-vanishing function. }The proof is not following Moser's path method in this case as the flow of the Reeb vector field is stationary.
	
	\begin{theorem}\label{contactDarbouxthm}
		Let $\alpha$ be a $b$-contact form inducing a $b$-contact structure $\xi$ on a $b$-manifold $(M,Z)$ of dimension $(2n+1)$ and $p\in Z$. We can find a local chart $(\mathcal{U},z,x_1,y_1,\dots,x_n,y_n)$ centered at $p$ such that on $\mathcal{U}$ the hypersurface $Z$ is locally defined by $z=0$ and
		
		\begin{enumerate}
			\item if $i(R_p) \neq 0$, \textcolor{black}{where $i$ denotes the inclusion given by Equation \ref{eq:binclusion}},
			\begin{enumerate}
				\item  and $\xi_p$ is singular, then $$\alpha|_\mathcal{U}=dx_1 +y_1\frac{dz}{z}+ \sum_{i=2}^n x_i dy_i,$$
				\item  and $\xi_p$ is regular, then $$\alpha|_\mathcal{U}=dx_1 +y_1\frac{dz}{z}+\frac{dz}{z}+ \sum_{i=2}^n x_i dy_i,$$
			\end{enumerate}
			\item if  $i(R_p)=0$, then $\tilde{\alpha}=f\alpha$ for $f(p)\neq 0$, where $$\tilde{\alpha}_p=\frac{dz}{z} +  \sum_{i=1}^n x_i dy_i.$$
		\end{enumerate}
	\end{theorem}
	
	We call the $b$-contact form \textit{regular} at $p\in M$ when the Reeb vector field is not vanishing at $p$ and \textit{singular} otherwise. \textcolor{black}{The proof we give here of the Darboux theorem for $b$-contact forms relies on the Darboux theorem for $b$-symplectic forms, by considering a contractible hypersurface transverse to the Reeb vector field and follows the similar arguments as in \cite{Vogel}. Alternatively, the theorem can be proved by constructing local coordinates and use Moser's path method, following closely the standard proof of the Darboux theorem for contact forms as in \cite{Geiges}.}

		\begin{proof}
	Let $p \in  Z$ and assume that $i(R_p) \neq 0$. It follows that $i(R_p)\neq 0$ in an open neighbourhood around $p$. The Reeb flow preserves the contact form as $\mathcal{L}_{R_\alpha}\alpha=0$. Let $H$ be a hypersurface transverse to the Reeb vector field. As $R_\alpha$ is tangent to $Z$ it follows that $H$ is transverse to $Z$. By transversality, $(H,d\alpha|_H)$ is a $b$-symplectic manifold. By the $b$-Darboux theorem for $b$-symplectic manifolds (Theorem 37 in \cite{GMP}), there exists a coordinates $(y,z,\textcolor{black}{x_i,y_i})$ on $H$ such that $d\alpha|_H=dy\wedge \frac{dz}{z}\textcolor{black}{+\sum_{i=2}^{n}dx_i\wedge dy_i}$. Here $z$ is just the defining coordinate for $Z$ restricted to $H$. Then $\sigma:= \alpha|_H-y\frac{dz}{z}\textcolor{black}{-\sum_{i=2}^{n}x_idy_i}$ is a closed $b$-form of degree $1$ on $H$. We now apply Poincaré lemma for $b$-forms \textcolor{black}{(see for instance Definition 3 in \cite{gmps} or Definition 9 in \cite{kms} and the subsequent discussion therein).} We assume here that $H$ is contractible. There are two subcases: either $\sigma$ is a smooth $1$-form or $\sigma$ is an honest $b$-form of degree $1$.
	
	In the first case, the Poincaré lemma implies that $\sigma= ds$ for some smooth function $s\in C^\infty(H)$. Choose $\epsilon>0$ such that the time-$t$-flow $\phi_t$  of $R_\alpha$ is defined for $t\in (-\epsilon,\epsilon)$ on a neighbourhood of $p$. Let
	\begin{align*}
	    \psi:(-\epsilon,\epsilon) \times H &\to M \\
	    (x,(y,z\textcolor{black}{,x_i,y_i})) &\mapsto \phi_x((y,z\textcolor{black}{,x_i,y_i})).
	\end{align*}
	The image of $\psi$ is a neighbourhood of $p\in M$ because $R_\alpha$ is transversal to $H$. By the implicit function theorem, $\psi$ defines a system of local coordinates on some open neighbourhood $U$ of $p$ in $M$ and by the definition of $\psi$, we have that $\frac{\partial}{\partial x}=R_\alpha$. The expression of $\alpha$ is invariant under the flow of $R_\alpha$. Let us denote the projection of $U$ to $H$ along the flow lines of $\phi_t$ by $pr$. Hence $\alpha$ in the given coordinates is given by
	$$\alpha=dx+ y \frac{dz}{z}\textcolor{black}{+\sum_{i=2}^nx_idy_i}+pr^*(ds)=d(x+s\circ pr) +y\frac{dz}{z}\textcolor{black}{+\sum_{i=2}^nx_idy_i}.$$
	Since $s\circ pr$ does not depend on the coordinate $x$, $x'=x+s\circ pr$ is a change of coordinate and in the local system of coordinates we obtain that $$\alpha=dx'+y\frac{dz}{z}\textcolor{black}{+\sum_{i=2}^nx_idy_i}.$$
	At the point $p$, the $b$-contact distribution $\ker \alpha$ is singular.
	
	In the second case, $\sigma$ is a honest $b$-form of degree $1$. By the Poincaré lemma for $b$-forms, $\sigma$ is the differential for a $b$-function defined on $H$, given by $s+k\log|z|$, where $s$ is a smooth function and $k$ is a constant $k\in \mathbb{R}\setminus \{0\}$. The proof goes along the same lines as in the previous subcase, namely flowing along the Reeb vector field, which preserves the coordinate $z$. In the local coordinates, $\alpha$ is given by the following expression
	$$\alpha=x+ y \frac{dz}{z}\textcolor{black}{+\sum_{i=2}^nx_idy_i}+pr^*(ds+k\frac{dz}{z})=d(x+s\circ pr) +y\frac{dz}{z}+k\frac{dz}{z}\textcolor{black}{+\sum_{i=2}^nx_idy_i}.$$
	By the same reasoning as before, the change $x'=x+s\circ pr$, $y'=y+k$ is a change of coordinates. In the new coordinates, we obtain that
	$$\alpha=dx+y\frac{dz}{z}+\frac{dz}{z}\textcolor{black}{+\sum_{i=2}^nx_idy_i}.$$
	At the point $p$, the $b$-contact distribution $\ker \alpha$ is regular.

		We finally consider the case where $i(R_p)=0$, which corresponds to the case where $d\alpha$ is a smooth de Rham form. By Lemma \ref{decomposition}, a $b$-form decomposes as $f\frac{dz}{z}+\beta$, where $z$ is a defining function for $Z$. As $d\alpha$ is smooth, \textcolor{black}{the function $f$ is given by $f(z,x_i,y_i)=a+zh(z,x_i,y_i)$ where $a$ is a constant and $h$ a smooth function. As $f(p) \neq 0$ (otherwise we would be in the smooth case), $a \neq 0$.} We choose a neighbourhood $\mathcal{U}$ around $p$ such that $f$ is non-vanishing on that neighbourhood. By dividing by $f$, the $b$-form $\tilde{\alpha}=\frac{dz}{z}+\tilde{\beta}$ defines the same distribution.
		Now take a $2n$-dimensional disk $D^{2n} \ni p$ in $\mathcal{U}$, \textcolor{black}{transverse to $z\frac{\partial}{\partial z}$. Then $d\alpha$ is symplectic on $D^{2n}$ due to the $b$-contact condition.} As $(D^{2n},d\tilde{\alpha})$ is symplectic, we know by applying Darboux theorem for symplectic forms (we assume the disk $D^{2n}$ small enough), that there exist $2n$ functions $x_i,y_i$ such that locally $d\tilde{\alpha}=\sum_{i=1}^n dx_i\wedge dy_i$. Now consider the $b$-form $\tilde{\alpha}-\sum_{i=1}^n x_idy_i-\frac{dz}{z}$. This form is closed and smooth. Hence by the Poincar\'e lemma for smooth forms, there exists a smooth function $g$ such that
		$$\tilde{\alpha}=\frac{dz}{z}+dg+\sum_{i=1}^n x_i dy_i.$$
		We can change the defining function to $\tilde{z}=e^{g}z$, so that $\frac{d\tilde{z}}{\tilde{z}}=\frac{dz}{z}+dg$. Now $$\tilde{\alpha}=\frac{d\tilde{z}}{\tilde{z}}+\sum_{i=1}^n x_i dy_i.$$
	As $\tilde{\alpha}\wedge(d\tilde{\alpha})^n=n\frac{d\tilde{z}}{\tilde{z}}\wedge dx_1\wedge dy_1 \wedge \dots \wedge dx_n \wedge dy_n \neq 0$, the functions $(\tilde{z},x_1,\dots,,x_n,y_1,\dots,y_n)$ form a local system of coordinates.

		\end{proof}	
	
	\begin{remark} \label{rmkregularopen}
		It follows from the $b$-Darboux theorem that if $(M,\ker \alpha)$ be a $b$-contact manifold and $\ker \alpha_p$ is regular for $p\in Z$, then there is an open neighbourhood around $p$ where $\ker \alpha$ is regular. Similarly, it follows that the set where the distribution $\ker\alpha$ is singular is of codimension $1$ in $Z$.
	\end{remark}

	The following example shows that it is possible to have both local models appearing on one connected component of the critical set. Furthermore, it shows in the case where the Reeb vector field is singular, we can only prove the normal form pointwise and the result does not hold in a local neighbourhood as in the case when the Reeb vector field is regular.

	\begin{example}\label{Exampleonepointsingular}
		$(S^2\times S^1, \alpha=\sin \varphi d\theta+\cos\varphi \frac{dh}{h})$ where $(\theta,h)$ are the polar coordinates on $S^2$ and $\varphi$ the coordinate on $S^1$. The Reeb vector field is given by $R_\alpha= \sin \varphi \frac{\partial}{\partial \theta}+\cos \varphi h\frac{\partial}{\partial h}$.
	\end{example}

	 We will prove that there are at least two points where the Reeb vector field is singular in the compact, $3$-dimensional case. This will be a corollary of the following. By definition of the $b^m$-tangent bundle, the Reeb vector field is tangent to the critical set. We can prove that in dimension $3$, the Reeb vector field is in fact Hamiltonian with respect to the induced area form from the contact condition. We will prove the following theorem:
	
	\begin{theorem}
		Let $(M,\alpha=u\frac{dz}{z}+\beta)$ be a $b$-contact manifold of dimension $3$. Then the restriction to $Z$ of the $2$-form $\Theta= u d\beta+\beta\wedge du$ is symplectic and the Reeb vector field is Hamiltonian with respect to $\Theta$ with Hamiltonian function $u$, i.e. $\iota_R \Theta= du$.
	\end{theorem}
	
	\begin{proof}
		In the decomposition, $\alpha$ is given by $\alpha= u\frac{dz}{z}+\beta$. The contact condition implies that $\Theta := ud\beta+\beta \wedge du$ is an area form, so it is symplectic. In the same decomposition, let us write the Reeb vector field as $R_\alpha= g \cdot z \frac{\partial}{\partial z}+X$, where $g \in C^\infty(M)$ and $ X\in \mathfrak{X}(Z)$. As $R_\alpha$ is the Reeb vector field, we obtain the following equations:
		\begin{align*}
		g\cdot u +\beta(X)=1, \\
		-g du +\iota_X d\beta=0, \\
		\iota_X du =0.
		\end{align*}
		A straightforward computation using those equations yields that $\iota_{X} \Theta= du$, hence the restriction of $R_\alpha$ to $Z$ is the Hamiltonian vector field for the function $-u$.
	\end{proof}
	
	In the compact case, we obtain:
	
	\begin{corollary}\label{corHam3dim}
		Let $(M,\alpha)$ be a $3$-dimensional compact $b$-contact manifold. Then there are at least two points where the local normal form of $\alpha$ is described by the singular model of the Darboux theorem.
	\end{corollary}

	\begin{remark}
		As shown in Example \ref{unitball}, there is a $b$-contact structure on the unit disk under the pull-back under the Moebius transformation of the regular local model. It follows from the last corollary, that this example can not be compactified.
	\end{remark}

	\begin{example}
		As before, consider $(S^2 \times S^1, \alpha= \sin \varphi d\theta+\cos \varphi \frac{dh}{h})$. The Reeb vector field on the critical set is given by Hamiltonian vector field of the function $-\cos \varphi$ with respect to the area form $d\varphi \wedge d\theta$. Hence, on the critical set, the Reeb vector field vanishes when $\sin \varphi=0$ and there are no periodic orbits of the Reeb vector field on the critical set.
	\end{example}

	A well known result in contact geometry is Gray's stability theorem, asserting that on a closed manifold, smooth families of contact structures are isotopic. The proof uses Moser's path method that works well in $b$-geometry. One proves along the same lines the following stability result for $b$-contact manifolds.
	
	\begin{theorem}\label{thm:gray}
		Let $(M,Z)$ compact $b$-manifold and let $(\xi_t)$, $t\in [0,1]$ be a smooth path of $b$-contact structures. Then there exists an isotopy $\phi_t$ preserving the critical set $Z$ such that $(\phi_t)_* \xi_0 =\xi_t$, or equivalently,
		\begin{equation}\label{eq:tobederived}
				\phi_t^*\alpha_t=\lambda_t\alpha_0
		\end{equation}
		 for a non-vanishing function $\lambda_t$.
	\end{theorem}
	
	\begin{proof}
		Assume that $\phi_t$ is the flow of a time dependent vector field $X_t$.\textcolor{black}{The flow exists for all time by the assumption that the manifold is compact.} Deriving Equation \ref{eq:tobederived}, we obtain
		$$d\iota_{X_t}\alpha_t+\iota_{X_t}d\alpha_t+\dot{\alpha}_t=\mu_t \alpha_t$$
		where $\mu_t=\frac{\dot{\lambda_t}}{\lambda_t}\circ\phi_t^{-1}$. If $X_t$ belongs to $\xi_t$, the first term of the last equation vanishes and applying then the Reeb vector field yields
		$$\dot{\alpha_t}(R_{\alpha_t})=\mu_t.$$
		The equation given by	
		$$\iota_{X_t}d\alpha_t=\mu_t\alpha_t-\dot{\alpha_t}$$
		then defines $X_t$ because $(\mu_t\alpha_t-\dot{\alpha_t})(R_{\alpha_t})=0$. We integrate the vector field $X_t$ to find $\phi_t$ and as $X_t$ is a vector field, tangent to the critical set, the flow preserves it.
	\end{proof}
	
	The compactness condition is necessary as it is shown in the next example.
	
	\begin{example}
		Consider the path of $b$-contact structures on $\R^3$ given by $\ker \alpha_t$ where $\alpha_t= (\cos \frac{\pi}{2} t -y\sin \frac{\pi}{2} t)\frac{dz}{z}+(\sin \frac{\pi}{2} t +y\cos \frac{\pi}{2} t)dx$. As $\alpha_0= \frac{dz}{z}+y dx$ and $\alpha_1= dx-y\frac{dz}{z}$, the two $b$-contact structures cannot be isotopic. This of course does not contradict  Theorem \ref{thm:gray} as this example is non-compact.
	\end{example}
	
	Along the same lines, we prove the following semi-local result.
	
	\begin{theorem}\label{semi-local}
		Let $(M,Z)$ be a $b$-manifold and assume $Z$ compact. Let $\xi_0=\ker \alpha_0$ and $\xi_1=\ker \alpha_1$ be two $b$-contact structures such that $\alpha_0|_Z=\alpha_1|_Z$. Then there exists a local isotopy $\psi_t$, $t\in [0,1]$ in an open neighbourhood $\mathcal{U}$ around $Z$ such that $\psi^*_t\alpha_t=\lambda_t \alpha_0$ and $\psi_t|_Z=\Id$ where $\lambda_t$ is a family of non-vanishing smooth functions.
	\end{theorem}
	
	\begin{proof}
		The proof is done following Moser's path method. Put $\xi_t=(1-t)\xi_0+t\xi_1$, $t\in [0,1]$. Because the non-integrability condition is an open condition and $\xi_t|_Z=\xi_0|_Z=\xi_1|_Z$, there exists an open neighbourhood $\mathcal{U}$ containing $Z$ such that $\xi_t$ is a family of $b$-contact structures. We will prove that there exists an isotopy $\psi_t:\mathcal{U} \mapsto M$ such that  $\psi_t^*\alpha_t=\lambda_t \alpha_0$, where $\lambda_t$ is a non-vanishing smooth function and $\lambda_t|_Z=1$. Assume that $\psi_t$ is the flow of a vector field $X_t$ and differentiating, we obtain the following equation:
		$$d\iota_{X_t}\alpha_t+\iota_{X_t}d\alpha_t+\dot{\alpha}_t= \mu_t \alpha_t,$$ where $\mu_t= \frac{d}{dt}(\log|\lambda_t|)\circ \psi_t^{-1}$.
		Taking $X_t \in \xi_t$, this equation writes down
		\begin{equation}\label{stabilityeq}
		\dot{\alpha}_t+\iota_{X_t}d\alpha_t=\mu_t \alpha_t.
		\end{equation}
		Applying the Reeb vector field to both sides, we obtain the equation that defines $\mu_t$:
		$$\mu_t=\dot{\alpha}_t(R_{\alpha_t}).$$ As $\dot{\alpha}_t|_Z=0$, $\mu_t|_Z=0$ and hence $X_t$ is zero on $Z$.
		By non-degeneracy of $d\alpha_t$ on $\xi_t$ there exists a unique $X_t\in \xi_t$ solving Equation \ref{stabilityeq}. Integrating $X_t$\textcolor{black}{, which we can do as $Z$ is compact, }yields the desired result.
	\end{proof}
	Note that this proof fails if one wants to prove stability of $b$-contact forms, that is we cannot assume that $\lambda_t= 1$ in a neighbourhood of $Z$.

	\section{$b$-Jacobi manifolds}\label{bJacobi}
	In the symplectic case, it is often helpful to look at $b$-symplectic manifolds as being the dual of a particular case of Poisson manifold. In contact geometry, Jacobi manifolds play this role.
	
	Recall that a Jacobi structure on a manifold $M$ is a triplet $(M,\Lambda,R)$ where $\Lambda$ is a smooth bi-vector field and $R$ a vector field satisfying the following compatibility conditions:
	\begin{align}
	[\Lambda,\Lambda]=2R \wedge \Lambda, && [\Lambda,R]=0,
	\end{align}
	where the bracket is the Schouten--Nijenhuis bracket. We refer the reader to \cite{Vaisman} and references therein for further information on Jacobi manifolds.

	\begin{definition}\label{defbJacobi}
		Let $(M,\Lambda,R)$ be a Jacobi manifold of dimension $2n+1$. We say that $M$ is a $b$-Jacobi manifold if $\Lambda^n\wedge R$ cuts the zero section of $\Lambda^{2n+1}(TM)$ transversally.
	\end{definition}

	Note that this definition is similar to the one of $b$-Poisson manifolds, in the sense that it also asks the top wedge power to be transverse to the zero section. We denote the hypersurface given by the zero section of $\Lambda^{2n+1}(TM)$ by $Z$ and we call it the \emph{critical set.}
	
	It is well-known that contact manifolds are a particular case of odd-dimensional Jacobi manifolds. A particular case of even-dimensional Jacobi manifolds are given by \emph{locally conformally symplectic manifolds.}
	
	\begin{definition}\label{def:locallyconformallysymplectic}
		A locally conformally symplectic manifold is a manifold $M$ of dimension $2n$ equipped with a non-degenerate two-form $\omega \in \Omega^2(M)$ that is locally \textcolor{black}{conformally} closed, which is equivalent to the existence of a closed $1$-form $\alpha \in \Omega^1(M)$ such that $d\omega=\alpha \wedge \omega$.
	\end{definition}
	
	Locally conformally symplectic manifold regained recent attention, notably in the work \cite{ChMu}.
	
	We will prove that $b$-contact manifolds and $b$-Jacobi manifolds are dual in some sense, as will be explained in the next two propositions. Before doing so, let us note that in the case where the dimension of the Jacobi manifold is $\dim M=2n$, we can given an similar definition to the one of Definition \ref{defbJacobi} by asking that $\Lambda^{2n}$ cuts the zero-section of $\Lambda^{2n}(TM)$ transversally. It should be possible to prove in the same lines that this case corresponds to locally conformally $b$-symplectic manifold.
	
%
%
%
%
		\begin{proposition}\label{prop:bcontacttobjacobi}
		Let $(M,\alpha)$ be a $b$-contact manifold. Let $\Lambda$ be the bi-vector field given by $\Lambda(df,dg)=d\alpha(X_f,X_g)$ where $X_f,X_g$ are the Hamiltonian $b$-vector fields associated to the smooth functions $f$ and $g$ (see Equation \ref{eq:Hamiltonianvf}) and let $R_\alpha$ be the Reeb vector field. Then $(M,\Lambda,R_\alpha)$ is a $b$-Jacobi manifold.
	\end{proposition}
	
	\begin{proof}
	\textcolor{black}{	As  $b$-Jacobi is a local condition, we can work in a local coordinate chart. Outside of the critical set, $\alpha$ is a contact form and the statement follows from the well-known fact that the associated Jacobi structure is defined by the bi-vector field $\Lambda$ defined by $\Lambda(df,dg)=d\alpha(X_f,X_g)$ and the Reeb vector field $R_\alpha$. Along the critical set, we use the local normal forms of $\alpha$ proved in Theorem \ref{contactDarbouxthm}. A straightforward computation now yields that for $\alpha=\frac{dz}{z}+\sum_{i=1}^n x_idy_i$,
	 $$\Lambda=\sum_{i=1}^n[\frac{\partial }{\partial x_i}\wedge \frac{\partial}{\partial y_i}+ z\frac{\partial}{\partial z}\wedge (x_i \frac{\partial }{\partial x_i})]$$
		and it follows from this expression that $[\Lambda,\Lambda]=2\Lambda \wedge R_\alpha$, $[\Lambda,R_\alpha]=0$ and $\Lambda^n\wedge R_\alpha \pitchfork 0$. The computation of the associated bi-vector field for the other local normal forms given in Theorem \ref{contactDarbouxthm} is done in a similar fashion.
		 Hence $(M,\Lambda,R_\alpha)$ is a $b$-Jacobi manifold.}
	\end{proof}
	
	Recall that to every Jacobi manifold $(M,\Lambda,R)$, one can associate a homogeneous Poisson manifold. Indeed, $(M\times \R,\Pi:=e^{-\tau}(\Lambda+\frac{\partial}{\partial\tau}\wedge R))$ is a Poisson manifold because
	\begin{align*}
	[\Pi,\Pi ]&= [ e^{-\tau}\Lambda,e^{-\tau}\Lambda]+2[e^{-\tau}\Lambda,e^{-\tau}\frac{\partial}{\partial \tau}\wedge R ]+[e^{-\tau}\frac{\partial}{\partial \tau} \wedge R,e^{-\tau}\frac{\partial}{\partial \tau}\wedge R]\\
	&=2e^{-2\tau}[\Lambda,\Lambda]+2(-e^{-\tau}\Lambda \wedge R)=0.
	\end{align*}
	Furthermore, the later is said to be \emph{homogeneous} because the vector field $T=\frac{\partial}{\partial\tau}$ satisfies
	$$\mathcal{L}_T \Pi=-\Pi.$$
	This construction is called Poissonization. The same stays true in the $b$-scenario, although we need to assume that the $b$-Jacobi manifold is of odd dimension, as $b$-Poisson manifold are defined only for even dimensions.
	\begin{lemma}
		The Poissonization of a $b$-Jacobi manifold of odd dimension is a homogeneous $b$-Poisson manifold.
	\end{lemma}
	\begin{proof}
		The proof is a straightforward computation:
		$$\Pi^{n+1}=-e^{-(n+1)\tau}\frac{\partial}{\partial\tau}\wedge\Lambda^n\wedge R.$$
		It follows from the definition of $b$-Jacobi that $\Pi$ is transverse to the zero-section.
	\end{proof}
	
	\begin{proposition}\label{bJacobicorrespondance}
		Let $(M^{2n+1},\Lambda,R)$ be a $b$-Jacobi manifold. Then $M$ is a $b$-contact manifold.
	\end{proposition}
	
	\begin{proof}
		The proposition is based on the local normal form of Jacobi structures, which are proved in \cite{DLM}. The main result is recalled in Appendix \ref{app:localJacobi}. Let $(M,\Lambda,R)$ be the $b$-Jacobi structure, so that $\Lambda^n \wedge R \pitchfork 0$. As usual, denote the critical hypersurface by $Z=(\Lambda^n \wedge R)^{-1}(0)$. First note that outside of $Z$, the leaf of the characteristic foliation is maximal dimensional. This is saying that outside of $Z$, the Jacobi structure is equivalent to a contact structure.
		
		Consider a point $p\in Z$ and denote the leaf of the characteristic foliation by $L$. By the transversality condition, the dimension of the leaf needs to be of dimension $2n$ or $2n-1$. Indeed, as $(M \times \R, e^{-\tau}(\frac{\partial}{\partial \tau}\wedge R +\Lambda))$ is $b$-Poisson, the critical set of $M\times \R$ is foliated by symplectic manifolds of codimension $2$, that is of dimension $2n$. Hence the critical set restricted to the hypersurface $\{\tau =0\}$, which is identified to be the critical set $Z$ of the initial manifold $M$, is foliated by codimension $1$ and codimension $2$ leaves.
		
		Let us first consider the case where at the point $x\in Z$, the leaf is of dimension $2n$. We will prove that this case corresponds to the case where the $R$ is singular, vanishing linearly. Let us apply Theorem 5.9 of \cite{DLM}. Hence the Jacobi manifold $(N,\Lambda_N,E_N)$ (see Theorem 5.9) is of dimension $1$, hence $\Lambda_N$ is zero. Hence $\Lambda$ is given by
		$$\Lambda= \sum_{i=1}^n \frac{\partial}{\partial x_i} \wedge \frac{\partial}{\partial x_{i+n}}-\sum_{i=1}^{n} x_{i+n}\frac{\partial}{\partial x_{i+n}}\wedge E_N.$$
		
		We now use the transversality condition on $\Lambda^n\wedge E_N$ to conclude that $E_N=z\frac{\partial}{\partial z}$ which is the same expression for the $b$-Jacobi structure associated to the $b$-contact form $\alpha=\frac{dz}{z}+\sum_{i=1}^n x_idx_{i+n}$.
		
		Let us consider the case where the leaf is of dimension $2n-1$. We will see that this corresponds to the case where the Reeb vector field is regular.
		According to Theorem 5.11 in \cite{DLM}, the bi-vector field is given by
		$$\Lambda=\Lambda_{2n-1}+ \Lambda_N+E\wedge Z_N$$
		where $(N,\Lambda_N,Z_N)$ is a homogeneous $2$-dimensional Poisson manifold and 
  
  $\Lambda_{2n-1}=\sum_{i=1}^{n-1}(x_{i+n-1}\frac{\partial}{\partial x_0}-\frac{\partial}{\partial x_i})\wedge \frac{\partial}{\partial x_{i+q}}$. The transversality condition implies that $\Lambda_{2n-1}^{n-1}\wedge\Lambda_N\wedge \frac{\partial}{\partial x_0}\pitchfork 0$, hence $\Lambda_N$ is a $b$-Poisson manifold. By \cite{GMP}, $\Lambda_N=z\frac{\partial}{\partial z} \wedge \frac{\partial}{\partial y}$. The homogeneous vector field $Z_N$ is determined by equation $\mathcal{L}_{Z_N}\Lambda_N=-\Lambda_N$. Hence $Z_N= y\frac{\partial}{\partial y}$. Hence the Jacobi structure is given by $E=\frac{\partial}{\partial x_0}$ and $$\Lambda= \sum_{i=1}^{n-1}(x_{i+n-1}\frac{\partial}{\partial x_0}-\frac{\partial}{\partial x_i})\wedge \frac{\partial}{\partial x_{i+q}}+z\frac{\partial}{\partial z} \wedge \frac{\partial}{\partial y}+ \frac{\partial}{\partial x_0}\wedge y \frac{\partial}{\partial y}, $$
		which is the Jacobi structure associated to the contact form $\alpha=dx_0+y\frac{dz}{z}+\sum_{i=1}^{n-1}x_idx_{i+q}$.
	\end{proof}
	
	\section{Geometric structure on the critical set}\label{structurecriticalset}
	
	To determine the induced structure of the $b$-contact structure on the critical set, we compute the associated Jacobi structure. Let us briefly review some results on Jacobi manifolds, which can all be found in \cite{Vaisman}. The Hamiltonian vector fields of a Jacobi manifold $(M,\Lambda,R)$ are defined by $X_f= \Lambda^\sharp(df)+fR$. It can be shown that the distribution $\mathcal{C}(M)=\{X_f|f\in C^\infty(M)\}$ is involutive and invariant under the Hamiltonian flow. The Stefan--Sussmann theorem asserts that $\mathcal{C}(M)$ integrates to a singular foliation, denoted by $\mathcal{F}$. As $\mathcal{C}(M)= \text{Im} \Lambda^\sharp +\langle R \rangle $, the leaves of $\mathcal{F}$ are even-dimensional when $R\in \text{Im} \Lambda^\sharp$ and odd dimensional in the other case. The induced structure on odd-dimensional leaves of $\mathcal{F}$ turns out to be a contact structure. For even dimensional leaves, one obtains locally conformally symplectic leaves. The definition of locally conformally symplectic manifolds is recalled in Definition \ref{def:locallyconformallysymplectic}.
	
	The computation of a Jacobi structure associated to a $b$-contact structure is \textcolor{black}{given as in Proposition \ref{prop:bcontacttobjacobi} by the bi-vector field defined by $\Lambda(df,dg)=d\alpha(X_f,X_g)$ and the Reeb vector field}. As we have proved a local norm form theorem, we can use the two local models to compute the associated Jacobi structure and check in both cases if $R\in \Lambda^\sharp$. We will prove
	
	\begin{theorem}
		Let $(M^{2n+1}, \alpha)$ be a $b$-contact manifold and $p\in Z$. We denote $\mathcal{F}_p$ the leaf of the singular foliation $\mathcal{F}$ going through $p$. Then
		\begin{enumerate}
			\item if $\xi_p$ is regular (as in Definition \ref{def:singularcontactstructure}), that is, $\mathcal{F}_p$ is of dimension $2n$, then the induced structure on $\mathcal{F}_p$ is locally conformally symplectic;
			\item if $\xi_p$ is singular (as in Definition \ref{def:singularcontactstructure}), that is, $\mathcal{F}_p$ is of dimension $2n-1$, then the induced structure on $\mathcal{F}_p$ is contact.
		\end{enumerate}
	\end{theorem}
	
	\begin{proof}
		By Theorem \ref{contactDarbouxthm}, if $\xi_p=\ker \alpha_p$ is singular, the Reeb vector field is not singular and the contact form can be written locally as
		$\alpha=dx_1+y_1\frac{dz}{z}+\sum_{i=2}^{n} x_i dy_i$. The Reeb vector field is given by $R=\frac{\partial}{\partial x_1}$,
		and the computation of the associated bi-vector field yields
		$$\Lambda=z\frac{\partial }{\partial y_1}\wedge \frac{\partial}{\partial z}+ \sum_{i=2}^{n} \frac{\partial}{\partial x_i}\wedge \frac{\partial}{\partial y_i}+\sum_{i=1}^{n}y_i\frac{\partial}{\partial x_1}\wedge\frac{\partial}{\partial y_i}.$$
		
		Hence, on the critical set, the bi-vector field is given by
		$$\Lambda\big|_Z=\sum_{i=2}^{n} \frac{\partial}{\partial x_i}\wedge \frac{\partial}{\partial y_i}+\sum_{i=1}^{n}y_i\frac{\partial}{\partial x_1}\wedge\frac{\partial}{\partial y_i}.$$
		Let us check if we can find a one form $\alpha$ such that $\Lambda\big|_Z^\#(\alpha)=\frac{\partial}{\partial x_1}$.
		For $y_1=0$, this cannot be solved, hence the set $\{z=0,y_1=0\}$ is a leaf with an induced contact structure.
		
		If $\xi_p$ is not singular and the Reeb vector is regular, the contact form can be written locally as $\alpha=dx_1+y_1\frac{dz}{z}+\frac{dz}{z}+\sum_{i=2}^{n} x_i dy_i$. A direct computation implies that the Reeb vector field lies in the distribution spanned by the bi-vector field $\Lambda$, hence the $b$-contact structure induces a locally conformally symplectic structure on the set $\{z=0,y_1\neq 0\}$.

		Last, if $\xi_p$ is not singular and the Reeb vector is singular, Theorem \ref{contactDarbouxthm} yields that the Reeb vector field can be written as $z\frac{\partial}{\partial z}$. As the Reeb vector field is vanishing, the critical set equals the $2n$-dimensional leaf spanned by $\text{Im}\Lambda^\sharp$. The induced structure on $\mathcal{F}_p$ is locally conformally symplectic.
	\end{proof}
	
	\begin{remark}
		Let us consider the case where $\dim M =3$ and the distribution $\xi$ is singular. Then the induced structure on the critical set is given by $\Lambda|_Z= y_1 \frac{\partial}{\partial y_1}\wedge\frac{\partial}{\partial x_1}$. As the critical set is a surface, it is clear that this is a Poisson structure and furthermore, that it is transverse to the zero section. Hence we obtain an induced $b$-symplectic structure on the critical set. Note that this is not true for higher dimensions.
	\end{remark}
	
	\begin{example}
	\textcolor{black}{$b$-Jacobi manifolds appear naturally in the study of geodesics of Lorentzian, or more generally pseudo-Riemannian manifolds.
	Given a manifold $M$ equipped with a pseudo-Riemannian metric $g$, denote by $\mathcal{L}$ the space of all oriented non-parametrized geodesics. The space $\mathcal{L}$ splits as $\mathcal{L}_{\pm}$, the space of space- and time-like geodesics, which are the geodesics $\gamma$ such that $g(\dot{\gamma},\dot\gamma)>0$ respectively $g(\dot{\gamma},\dot\gamma)<0$ , and $\mathcal{L}_0$, the space of light-like geodesics, which corresponds to $g(\dot\gamma,\dot\gamma)=0$. As was shown in \cite{khta}, $\mathcal{L}_\pm$ is even dimensional and admits an induced symplectic structure, obtained by Hamiltonian reduction on the level-set of the Hamiltonian $H=\pm1$ defined on the cotangent bundle of $M$, equipped with the standard symplectic structure.
	  The space $\mathcal{L}_0$ can be seen as the common boundary of $\mathcal{L}_\pm$ and has an induced contact structure.  These structures can be regrouped into a Jacobi structure on $\mathcal{L}$, having two leaves of maximal dimension given by $\mathcal{L}_+$ and $\mathcal{L}_-$ and one leaf of codimension $1$ given by $\mathcal{L}_0$. In the special case when $\dim \mathcal{L}=2$, this is indeed a $b$-Poisson structure (see Remark 2.8 in \cite{khta}). The described structure of the space of all oriented non-parametrized geodesics thus resembles the induced structure of a $b$-contact structure on the critical set.}
	\end{example}
	
	\section{Symplectization and contactization}\label{symplectization}
	
	Symplectic and contact manifolds are related to each other as follows. It is well-known that a contact manifold can be transformed into a symplectic one by \emph{symplectization}: if $(M,\alpha)$ is a contact manifold, then $(M\times \R, d(e^t\alpha))$ (where $t$ is the coordinate on $\R$) is a symplectic manifold. On the other hand, hypersurfaces of a symplectic manifold $(M,\omega)
	$ are contact, provided that there exists a vector field satisfying $\mathcal{L}_X \omega= \omega$ that is transverse to the hypersurface. Such a vector field is called Liouville vector field. The contact form on the hypersurface is given by the contraction of the symplectic form with the Liouville vector field, i.e. $\alpha= \iota_X \omega$.
	
	We will show that the same holds in the $b$-category.
	
	\begin{example}
		Let $(\R^4,\omega=\frac{1}{z}dz \wedge dt + dx \wedge dy)$ be a $b$-symplectic manifold. A Liouville vector field is given by $X=\frac{1}{2}(z\frac{\partial}{\partial z} +2t \frac{\partial}{\partial t}+ x \frac{\partial}{\partial x} +y\frac{\partial}{\partial y})$. Note that Liouville vector fields are defined up to addition of symplectic vector fields, that is a vector field $Y$ satisfying $\mathcal{L}_Y \omega=0$. Another Liouville vector field is for example given by $t\frac{\partial}{\partial t}+x\frac{\partial}{\partial x}$.
	\end{example}
	
	Let us take a $b$-symplectic manifold $(W,\omega)$ of dimension $(2n+2)$ and a Liouville vector field $X$ on $W$ that is transverse to a hypersurface $H$ of $W$. Then $(H,\iota_X \omega)$ is a $b$-contact manifold of dimension $(2n+1)$ as $\iota_X \omega \wedge (d\iota_X\omega)^{n}=\frac{1}{n+1}\iota_X(\omega^{n+1})$ is a volume form provided that $X$ is transverse to $H$. If $H$ does not intersect the critical set, one obtains of course a smooth contact form. Due to the $b$-Darboux theorem, there are two local models for $b$-contact manifolds and we will see that we can obtain both structures, depending on the relative position of the hypersurface with the Reeb vector field on it.
	
	
	\begin{example} \label{ex:bgeodesic}
		Let us take $(W=\R^4,\omega=\frac{1}{z}dz \wedge dt + dx \wedge dy)$ and the Liouville vector field $X=t\frac{\partial}{\partial t}+x\frac{\partial}{\partial x}$. The contraction of $X$ with the $b$-symplectic form yields $\iota_X \omega=-\frac{t}{z}dz +xdy$. Let us take different hypersurfaces transverse to $X$ and compute the induced $b$-contact form.
		\begin{itemize}
			\item   If we take as hypersurface the hyperplane $M_1=\{(1,y,-t,z),y,t,z\in \R\}$, which is transverse to $X$, we obtain $\alpha=dy+t\frac{dz}{z}$, which is the regular local model.
			\item   If we take as hypersurface the hyperplane $M_2=\{(x,y,-1,z),x,y,z\in \R\}$, which is transverse to $X$, we obtain $\alpha=\frac{dz}{z}+xdy$, which is the singular local model.
		\end{itemize}
	\end{example}
	
    \begin{example}\label{example:S3}
		The three dimensional sphere admits a $b$-contact structure. Consider the $\R^4$ with the standard $b$-symplectic structure $\omega=\frac{dx_1}{x_1}\wedge dy_1 + dx_2\wedge dy_2$ and denote by $S^3$ the unit sphere in $\R^4$. The Liouville vector field $X=\frac{1}{2}x_1\frac{\partial}{\partial x_1}+y_1\frac{\partial}{\partial y_1}+\frac{1}{2}(x_2 \frac{\partial}{\partial x_2}+y_2\frac{\partial }{\partial y_2})$ is transverse to the sphere and hence $\iota_X \omega$ defines a $b$-contact form on $S^3$. The critical set is a $2$ dimensional sphere, $S^2$, given by the intersection of the sphere with the hyperplane $z=0$.
	\end{example}
	
	\begin{example}\label{ex:bgeod}
		The unit cotangent bundle of a $b$-manifold has a natural $b$-contact structure. Let $(M,Z)$ be a $b$-manifold of dimension $n$ with coordinates $z,y_i, i=2,\dots,n$ as in Example \ref{extendedphasespace}. It is shown in \cite{GMP} that the cotangent bundle has a natural $b$-symplectic structure defined by the $b$-form given by the exterior derivative $d\lambda=d\big(x_1\frac{dz}{z} +\sum_{i=2}^{n} x_i dy_i\big)$. The unit $b$-cotangent bundle is given by ${^bT^*_1M}=\{(z,y_2,\dots,y_n,x_1,\dots,x_n)\in {^bT^*M} | \  \sum_{i=1}^n x_i^2=1 \}$. The vector field $\sum_{i=1}^n x_i \frac{\partial}{\partial {x_i}}$ defined on the $b$-cotangent bundle ${^bT^*M}$ is a Liouville vector field, and is transverse to the unit $b$-cotangent bundle, and hence induces a $b$-contact structure on it.

 From a Riemannian point of view, the Reeb vector field describes the geodesic flow associated to a $b$-metric, that is a metric a bundle metric on ${^b}TM$. 
 It induces a bundle metric $g^*$ on ${^b}T^*M$. The unit cotangent bundle is alternatively described by ${^bT^*_1M}=\{X\in {^b}T^*M | g^*(X,X)=1\}$ and the associated Reeb vector field to the associated contact form as described above, is the push-forward under the bundle isomorphism of the geodesic vector field on ${^b}TM$.
	\end{example}
	
	We will compute a particular case of the unit cotangent bundle of a $b$-manifold.
	
	\begin{example}\label{example:3torus}
	Consider the torus $\mathbb{T}^2$ as a $b$-manifold where the boundary component if given by two disjoint copies of $S^1$. The unit cotangent bundle $S^*\mathbb{T}^2$, diffeomorphic to the $3$-torus $\mathbb{T}^3$ is a $b$-contact manifold with $b$-contact form given by $\alpha=\sin\phi\frac{dx}{\sin(x)}+\cos \phi dy$, where $\phi$ is the coordinate on the fiber and $(x,y)$ the coordinates on $\mathbb{T}^2$.
	\end{example}
	
	We saw that hypersurfaces of $b$-symplectic manifolds that are transverse to a Liouville vector field have an induced $b$-contact structure. The next lemma describes which model describes locally the $b$-contact structure.
	
	\begin{lemma}
		Let $(W,\omega)$ be a $b$-symplectic manifold and $X$ a Liouville vector field transverse to a hypersurface $H$. Let $R$ be the Reeb vector field defined on $H$ for the $b$-contact form $\alpha=i_X \omega$. Then $R \in H^\perp$, where $H^\perp$ is the symplectic orthogonal of $H$.
	\end{lemma}
	
	\begin{proof}
		The Reeb vector field defined on $H$ satisfies $i_R(d\alpha)|_H=i_R (di_X\omega)|_H=i_R\omega|_H=0$.
	\end{proof}
	
	Hence if $H^\perp$ is generated by a singular vector field, the contact manifold $(H,\alpha)$ is locally of the second type as in the $b$-Darboux theorem. In the other case, the local model is given by the first type.

	We now come back to the contactization of a $b$-symplectic manifold.
	
	\begin{theorem}
		Let $(M,\alpha)$ be a $b$-contact manifold. Then $(M\times \R, \omega=d(e^t\alpha))$ is a $b$-symplectic manifold.
	\end{theorem}
	
	\begin{proof}
		It is clear that $\omega$ is a closed $b$-form. Furthermore, a direct computation yields
		$$\big((e^td\alpha)\big)^{n+1}=e^{t(n+1)}dt\wedge \alpha \wedge (d\alpha)^n,$$
		which is non-zero as a $b$-form by the non-integrability condition.
	\end{proof}
	
	It is easy to see that $\frac{\partial}{\partial t}$ is a Liouville vector field of the symplectization $(M\times \R,d(e^t\alpha))$, which is clearly transverse to the submanifold $M \times \{0\}$. Hence, we obtain the initial contact manifold $(M,\alpha)$. This gives us the following proposition.

	\begin{proposition}
		Every $b$-contact manifold can be obtained as a hypersurface of a $b$-symplectic manifold.
	\end{proposition}
	
	\begin{remark}
		Another close relation between the symplectic and the contact world is the contactization: take an exact symplectic manifold, i.e. $(M,d\beta)$, then $(M\times \R,\beta+dt)$, where $t$ is the coordinate on $\R$, is contact. This remains true in the $b$-case. Furthermore, it is clear that by this construction, we obtain $b$-contact forms of the first type, as the Reeb vector field is given by $\frac{\partial}{\partial t}$.
	\end{remark}
	
	\section{Other singularities}\label{sec:othersingularities}

	In what follows, we consider contact structures with higher order singularities. Let $(M^n,Z)$ be a manifold with a distinguished hypersurface and let us assume that $Z$ is the zero level-set of a function $z$. The $b^m$-tangent bundle, which we denote by ${^{b^m}}TM$, can be defined to be the vector bundle whose sections are generated by
	$$\Big\{z^m\frac{\partial}{\partial z},\frac{\partial}{\partial x_1},\dots, \frac{\partial}{\partial x_{n-1}}\Big\}.$$
	The de Rham differential can be extended to this setting. The notion of $b$-symplectic manifolds then naturally extends and we talk about $b^m$-symplectic manifolds, see \cite{Scott,gmw1}. {In the same fashion, we can extend the notion of $b$-tangent and $b$-cotangent bundle, and thus also $b$-contact manifolds to the $b^m$-setting}\footnote{We remark here that further generalizations as contact structures over different Lie algebroids are possible.  The symplectic analogue is being treated in \cite{mirandascott}.} we say that a $b^m$-form $\alpha \in {^{b^m}}\Omega^n(M)$ is $b^m$-contact if $\alpha \wedge (d\alpha)^n \neq 0$ where the dimension of $M$ is $2n+1$. The proofs of the theorems of the previous sections, in particular Theorem \ref{contactDarbouxthm} and Proposition \ref{bJacobicorrespondance} and the construction carried out in Section \ref{symplectization}, generalize directly to this setting. For the sake of a clear notation, we do not write down the statements of the generalization, but only assert informally, that $b$ can be replaced by $b^m$ in the statements.
	
	In the next section, we are going to relate the topology of $b^m$-contact structures to smooth contact structures, but also a generalization of confoliations that are given by the following definition.
	
	\begin{definition}\label{def:foldedcontact}
	A smooth $1$-form $\alpha \in \Omega^1(M^{2n+1})$ is a \emph{folded contact form} if $\alpha \wedge (d\alpha)^n$ intersect the zero-section of the bundle $\bigwedge^{2n+1} T^*M$ transversally.
	\end{definition}

	 By the transversality, the set where the contact condition fails is given by a hypersurface, that we call the \emph{folding hypersurface}. 	As mentioned before, this is a a generalization of confoliation, see \cite{et}, where the $1$-form is asked to satisfy $\alpha \wedge (d\alpha)^n \geq 0$ (for positive confoliation).
	
	We do not enter in a more detailed study of folded contact forms but only remark that
folded contact structures appear in the work of \cite{martinet2,jazh, jz}, \textcolor{black}{where folded contact forms appeared under the name of \emph{singular contact form with structurally smooth hypersurface}}.
	
	\section{Desingularization of $b^m$-contact structures}\label{sec:desingularization}

	In this section, we desingularize singular contact structures and consequently explain the relation to smooth contact structures. The proof is based on the idea of \cite{gmw1}. However, in contrast to the symplectic case, we need an additional assumption in order to desingularize the $b^m$-contact form.
	
	Recall that from Lemma \ref{decomposition}, it follows that a $b^m$-form $\alpha \in {^{b^m}}\Omega^1(M)$ decomposes $\alpha=u\frac{dz}{z^m}+\beta$ where $u\in C^{\infty}(M)$ and $\beta \in \Omega^1(M)$. In order to desingularize the $b^m$-contact forms, we will assume that $\beta$ is the pull-back under the projection of a one-form defined on $Z$.
	
	\begin{definition}\label{def:bcontactconvex}
		We say that a $b^m$-contact form $ \alpha$ is almost convex if $\beta= \pi^*\tilde{\beta}$, where $\pi:\mathcal{N}(Z) \to Z$ is the projection from a tubular neighbourhood of $Z$ to the critical set and $\tilde{\beta}\in \Omega^1(Z)$. We will abuse notation and write $\beta \in \Omega^1(Z)$. We say that a $b^m$-contact \textcolor{black}{form} is convex if $\beta \in \Omega^1(Z)$ and $u \in C^\infty(Z)$.
	\end{definition}
	
	Note that the this notion is to be compared to the one of convex hypersurfaces, which we will recall in the next section. As we will see in the next lemma, almost convex $b^m$-contact structures are semi-locally isotopic to convex ones.
	
	\begin{lemma}
		Let $(M, \alpha)$ be an almost convex $b^m$-contact manifold and let the critical hypersurface $Z$ be compact. Then there exists a neighbourhood around the critical set $\mathcal{U}\supset Z$, such that $\alpha$ is isotopic to a convex $b^m$-contact form relative to $Z$ on $\mathcal{U}$.
	\end{lemma}
	
	\begin{proof}
		Let $\alpha= u\frac{dz}{z^m}+\beta$ where $u \in C^\infty(M)$ and $\beta \in \Omega^1(Z)$. Put $\tilde{\alpha}=u_0\frac{dz}{z^m}+\beta$, where $u_0=u| _Z \in C^\infty(Z)$, which is convex. Take the linear path between the two $b^m$-contact structures, which is a path of $b^m$-contact structures because $\xi$ and $\tilde{\xi}$ equal on $Z$. Applying Theorem \ref{semi-local}, we obtain that there exist a local diffeomorphism $f$ preserving $Z$ and a non-vanishing function $\lambda$ such that on a neighbourhood of $Z$, $f^*{\alpha}=\lambda \tilde{\alpha}$.
	\end{proof}
	
	The next lemma gives intuition on this definition and gives a geometric characterization of the almost-convexity in terms of the $f_{\epsilon}$-desingularized symplectization.
	
	\begin{lemma}\label{convexLiouville}
		A $b^m$-contact manifold $(M, \alpha)$ is almost-convex if and only if the vector field $\frac{\partial}{\partial t}$ is a Liouville vector field in the desingularization of the $b^m$-symplectic manifold obtained by the symplectization of $(M, \alpha)$. Here $t$ denotes the coordinate of the symplectization.
	\end{lemma}
	
	\begin{proof}
		Let $(M, \alpha)$ be a almost-convex $b^m$-contact manifold. The symplectization is given by $(M\times \R, \omega= d(e^t\alpha))$. The desingularization technique of Theorem \ref{symplecticdesingularization} produces a family of symplectic forms $\omega_\epsilon=ue^t dt \wedge df_\epsilon+e^t dt\wedge \beta +e^t du \wedge df_\epsilon+e^t d\beta$. From almost-convexity it follows that $\frac{\partial}{\partial t}$ preserves $\omega_\epsilon$, so $\frac{\partial}{\partial t}$ is a Liouville vector field.
		
		To prove the converse, assume that $\frac{\partial}{\partial t}$ is a Liouville vector field in $(M,\omega_\epsilon)$. It follows from the fact that $\mathcal{L}_{\frac{\partial}{\partial t}}\omega_\epsilon=\omega_\epsilon$ that $\beta \in \Omega^1(Z)$.
	\end{proof}
	
	We will see that under almost-convexity, the $b^{2k}$-contact forms can be desingularized.
	
	\begin{theorem}\label{thm:desingularization}
		Let $(M^{2n+1}, \alpha)$ a $b^{2k}$-contact manifold with critical hypersurface $Z$. Assume that $\alpha$ is almost convex. Then there exists a family of contact forms $\alpha_\epsilon$ which coincides with the $b^{2k}$-contact form $\alpha$ outside of an $\epsilon$-neighbourhood of $Z$. The family of bi-vector fields $\Lambda_{\alpha_\epsilon}$  and the family of vector fields $R_{\alpha_\epsilon}$ associated to the Jacobi structure of the contact form $\alpha_\epsilon$ converges to the bivector field $\Lambda_{\alpha}$ and to the vector field $R_\alpha$ in the $C^{2k-1}$-topology as $\epsilon \to 0$.
	\end{theorem}

	We call $\alpha_\epsilon$ the $f_\epsilon$-desingularization of $\alpha$.
	
	A corollary of this is that almost-convex $b^m$-contact forms admit a family of contact structures if $m$ is even, and a family of folded-type contact structures is $m$ is odd.
	
	The proof of this theorem follows from the definition of convexity and makes use of the family of functions introduced in \cite{gmw1}.

	\begin{proof}
		By the decomposition lemma, $\alpha=u\frac{dz}{z^m}+\beta$. As $\alpha$ is almost convex, the contact condition writes down as follows:
		$$\alpha \wedge (d\alpha)^n= \frac{dz}{z^m}\wedge (u (d\beta)^n+n\beta\wedge du\wedge(d\beta)^{n-1})\neq0.$$
		In an $\epsilon$-neighbourhood, we replace $\frac{dz}{z^m}$ by a smooth form. The expression depends on the parity of $m$.
		
		Following \cite{gmw1} we consider an odd smooth function
			$f \in \mathcal{C}^\infty(\mathbb{R})$ satisfying $f'(x) > 0$ for all $x \in \left[-1,1\right] $ and satisfying outside that
			\begin{equation}
			f(x) = \begin{cases}
			\frac{-1}{(2k-1)x^{2k-1}}-2& \text{for} \quad x < -1,\\
			\frac{-1}{(2k-1)x^{2k-1}}+2& \text{for} \quad x > 1.\\
			\end{cases}
			\end{equation}
			Let $f_\epsilon(x)$ be defined as $\epsilon^{-(2k -1)}f(x/\epsilon)$.
		
		We obtain the family of globally defined $1$-forms given by $\alpha_\epsilon= udf_\epsilon+\beta$ that agrees with $\alpha$ outside of the $\epsilon$-neighbourhood. Let us check that $\alpha_\epsilon$ is contact inside this neighbourhood. Using the almost-convexity condition, the non-integrability condition on the $b^m$-form $\alpha$ writes down as follows:
		$$\alpha_\epsilon \wedge (d\alpha_\epsilon)^n = dz \wedge (f_\epsilon'(z)u d\beta+f_\epsilon'(z)\beta \wedge du -\beta \wedge\frac{\partial \beta}{\partial z} ).$$
		We see that $\alpha_\epsilon \wedge d\alpha_\epsilon= f_\epsilon'(z)z^m\alpha\wedge d\alpha$ and hence $\alpha_\epsilon$ is contact.
		
		We denote by $\Lambda_\alpha$ and $R_\alpha$ the bi-vector field and vector field of the $b$-contact form $\alpha$. Now let us check that the bi-vector field $\Lambda_{\alpha_\epsilon}$ and the vector field of $R_{\alpha_\epsilon}$ corresponding to the Jacobi structure of the desingularization converge to $\Lambda_\alpha$ and $R_\alpha$ respectively.
		
		Let us write $R_\alpha$ and $\Lambda_\alpha$ in a neighbourhood of a point $p\in Z$.
		
		$$R_\alpha=g z^{2k}\frac{\partial}{\partial z}+X, \quad \Lambda_\alpha= z^{2k}\frac{\partial}{\partial z} \wedge Y_1+Y_2 \wedge Y_3 $$
		where $g \in C^\infty(M)$ and $X,Y_i \in \mathfrak{X}(Z)$ for $i=1,2,3$. The Jacobi structure associated to the desingularization is given by
		
		$$R_{\alpha_\epsilon}=g \frac{1}{f'_\epsilon(z)} \frac{\partial}{\partial z}+X, \quad \Lambda_{\alpha_\epsilon}= \frac{1}{f'_\epsilon(z)}\frac{\partial}{\partial z} \wedge Y_1+Y_2 \wedge Y_3.$$
		The $C^{2k-1}$-convergence follows from this formulas.
	\end{proof}
	
	Similarily, we can prove a desingularization theorem for $b^{2k+1}$-contact forms.
	
	\begin{theorem}\label{thm:desingularizationcontactb2k+1}
		Let $(M^{2n+1}, \alpha)$ a $b^{2k+1}$-contact manifold with critical hypersurface $Z$. Assume that $\alpha$ is almost convex. Then there exists a family of folded contact forms $\alpha_\epsilon$ which coincides with the $b^{2k+1}$-contact form $\alpha$ outside of an $\epsilon$-neighbourhood of $Z$.
	\end{theorem}

	\begin{proof}
		By the decomposition lemma, $\alpha=u\frac{dz}{z^{2k+1}}+\beta$. As $\alpha$ is almost convex, the contact condition writes down as follows:
		$$\alpha \wedge (d\alpha)^n= \frac{dz}{z^{2k+1}}\wedge (u (d\beta)^n+n\beta\wedge du\wedge(d\beta)^{n-1})\neq0.$$
		In an $\epsilon$-neighbourhood, we replace $\frac{dz}{z^{2k+1}}$ by a smooth form.
		
		Following \cite{gmw1} we consider an even smooth function given by $f_\epsilon(x):=\frac{1}{\epsilon^{2k}}f(\frac{x}{\epsilon})$ where $f\in C^\infty(\R)$ satisfies
	\begin{itemize}
		\item $f>0$ and $f(x)=f(-x)$,
		\item $f'(x)>0$ if $x<0$,
		\item $f(x)=-x^2+2$ if $x\in [-1,1]$,
		\item $f(x)=\log(|x|)$ if $k=0, x\in \R\setminus [-2,2]$.
		\item $f(x)=\frac{-1}{(2k+2)x^{2k+2}}$ if $k>0$, $x\in \R \setminus [-2,2]$.
	\end{itemize}
	
	We define $\alpha_\epsilon=udf_\epsilon +\beta$. We see that $\alpha_\epsilon \wedge d\alpha_\epsilon= f_\epsilon'(z)dz\wedge(u(d\beta)^n+n\beta \wedge du \wedge (d\beta)^{n-1})$ and hence $\alpha_\epsilon$ is folded contact: indeed $f_\epsilon'$ vanishes transversally at zero, and away from zero, this last expression is non-zero.
		
	\end{proof}
	
	An alternative proof of this theorem would be to use the symplectization as explained in Section \ref{symplectization} and to use immediately Theorem \ref{symplecticdesingularization} in the symplectization. The almost convex condition makes sure that the vector field in the direction of the symplectization is Liouville in the desingularization, see Lemma \ref{convexLiouville}. Hence the induced structure is contact. Without the almost-convexity, the induced structure of the desingularized symplectic form on the initial manifold is not necessarily contact. This is saying that almost-convexity is a sufficient condition, but not a necessary condition to apply the desingularization method.

	In the next section of the article, we will see that in presence of convex hypersurface in contact manifolds, a construction that transforms the convex hypersurface into the critical set of a $b^m$-contact manifold, holds.
	
	\section{Existence of singular contact structures on a prescribed submanifold}\label{sec:existence}

	Existence of contact structures on odd dimensional manifolds has been one of the leading questions in the field. The first result in this direction was proved for open odd-dimensional manifolds by Gromov \cite{Gromov}. The case for closed manifolds turned out to be much more subtle. The $3$-dimensional case was proved by Martinet--Lutz \cite{lutz,martinet}. The existence problem for higher dimensionswas solved  the celebrated article by Borman--Eliashberg--Murphy \cite{bem}.
	\begin{theorem}
		Any almost contact closed manifold $M^{2n+1}$ admits a contact structure.
	\end{theorem}
	
	We give in this section an answer to the question whether or not closed manifolds also admit $b^m$-contact structures. The result relies on convex hypersurface theory, which was introduced by Giroux \cite{Giroux}. \textcolor{black}{Throughout the next two sections, the hypersurfaces in consideration will always be embedded hypersurfaces.}

	\begin{definition}
		Let $(M,\ker \alpha)$ be a contact manifold. A vector field $X$ is contact if it preserves $\xi$, that is $\mathcal{L}_X \alpha= g\alpha$ for $g\in C^\infty(M)$. An embedded hypersurface $Z$ in $M$ is convex if there exists a contact vector field $X$ that is transverse to $Z$.
	\end{definition}

	It follows from this definition that the contact {structure} can be written under vertically invariant form in a neighbourhood of $Z$, that is $\alpha=g\cdot (u dt +\beta)$, where the contact vector field $X$ is given by $\frac{\partial}{\partial t}$,  $u \in C^\infty(Z)$, $\beta \in \Omega^1(Z)$ and $g$ is a function that is not necessarily vertically invariant. Note that the definition of convex $b^m$-contact forms, that is Definition \ref{def:bcontactconvex}, is the analogue of this definition in the $b$-setting. As was proved by Giroux \cite{Giroux}, in dimension $3$, there generically all closed surfaces are convex.
	
	\begin{theorem}[\cite{Giroux}]
		Any $C^\infty$-generic closed surface  in a $3$-dimensional contact manifold is a convex surface.
	\end{theorem}
	
	In higher dimension, this result does not hold for generic hypersurfaces, see \cite{mori}. 
    {However in higher dimensions, a $C^0$-generic hypersurface is convex, as is proved in Theorem 1.1 in \cite{ep} (see also \cite{hh}). }

	In the theory of convex hypersurfaces, a fundamental role is played by the set $\Sigma$ given by the points of the convex hypersurface where the transverse contact vector field belongs to the contact distribution. It is a consequence of the non-integrability condition that $\Sigma$ is a codimension $1$ submanifold in $Z$. When $M$ is of dimension $3$, a connected component of $\Sigma$ is called the dividing curve. Loosely speaking, the dividing curves determine the germ of the contact structure on a neighbourhood of the convex surface. For a precise statement, see \cite{Giroux,Gi}.
	
	We will prove that convex hypersurfaces can be realized as the critical set of $b^{2k}$-contact structures. A similar result holds for $b^{2k+1}$-contact structures. However, in this case the critical set has two connected components, which correspond to two convex hypersurfaces arbitrarily close to a connected component of the given convex hypersurface.
	
	\begin{theorem}\label{Existencetheorem}
		Let $(M,\xi)$ be a contact manifold and let $Z$ be a convex hypersurface in $M$. Then $M$ admits a $b^{2k}$-contact structure for all $k$ that has $Z$ as critical set. The codimension $2$ submanifold $\Sigma$ corresponds to the set where the rank of the distribution drops and the induced structure is contact.
	\end{theorem}

 Before proving Theorem \ref{Existencetheorem}, we observe some corollaries.
	
	Using Giroux's genericity result in dimension $3$, we obtain the following corollary in dimension $3$:
	
	\begin{corollary}\label{cor:3dgeneric}
		Let $M$ be a $3$-dimensional manifold. Then for a $C^\infty$-generic surface $Z$, there exists a $b^{2k}$-contact structure on $M$ realising $Z$ as the critical set.
	\end{corollary}

		\begin{proof}[Proof of  {Corollary \ref{cor:3dgeneric}}]
		Using Gromov's result in the open case and Lutz--Martinet for $M$ closed, we can equip $M$ with a contact form.
		As is proved in \cite{Giroux}, a generic surface $Z$ is convex and the conclusion follows from Theorem \ref{Existencetheorem}.
	\end{proof}

    {In higher dimensions, the following holds.}

    \begin{corollary}\label{cor:highergeneric}
        {Let $M$ be an almost contact manifold of dimension $2n+1$, for $n>1$ and let $Z$ be a $C^0$-generic hypersurface in $M$. Then there exists a $b^{2k}$-contact structure on $M$ realising $Z$ as the critical set.}
    \end{corollary}

    \begin{proof}[Proof of Corollary \ref{cor:highergeneric}]
        {By the $h$-principle for contact structures in higher dimensions, as proved in \cite{bem}, the almost contact stucture can be homotoped to a contact structure. By Theorem 1.1 in \cite{ep}, a $C^0$-generic hypersurface is convex. We can thus apply Theorem \ref{Existencetheorem}.}
    \end{proof}

	\textcolor{black}{This corollary can be improved: instead of changing the given surface to a convex one but fixing the contact structure, we will fix the given surface and change the contact structure to realize the fixed surface as a convex \textcolor{black}{surface}. Theorem \ref{Existencetheorem} will then imply the existence of a $b^{2k}$-contact structure whose critical hypersurface is given by the given surface.}
	
	\begin{corollary}\label{cor:3dfixed}
	{Let $M$ be an almost contact manifold and $Z\subset M$ be a hypersurface. Then there exists a $b^{2k}$-contact structure on $M$ realising $Z$ as the critical set.}
	\end{corollary}

 			\begin{proof}[Proof of Corollary \ref{cor:3dfixed}]
		By \cite{bem}, we can equip $M$ with a contact structure $\xi=\ker \alpha$.
  		As is proved in \cite{Giroux} in the $3$-dimensional case (respectively \cite{ep} in the higher dimensional case), a $C^\infty$-small (respectively a $C^0$-small) perturbation of the hypersurfacesurface $Z$ is  convex and let us denote this perturbation by $Z_1$. Consider the isotopy $\varphi_t:M\to M$ such that $\varphi_1(Z)=Z_1$. This isotopy is supported in a tubular neighbourhood around the critical set $Z$. Consider the contact structure given by the kernel of $\varphi_1^*\alpha$. For this contact structure, the initial surface $Z$ is convex. The result now follows by applying Theorem	\ref{Existencetheorem} to {the} contact structure $(M,\ker \varphi^*\alpha)$ and the convex {hypersurface} $Z$.
	\end{proof}
	

 We will now proceed to the proof of Theorem \ref{Existencetheorem}.

	\begin{proof}[Proof of Theorem \ref{Existencetheorem}]
		Using the transverse contact vector field, we find a tubular neighbourhood of $Z$ diffeomorphic to $Z \times \mathbb{R}$ such that the contact structure is defined by the contact form $\alpha= u dt + \beta$, where $t$ is the coordinate on $\mathbb{R}$, $u \in C^\infty(Z)$ and $\beta \in  \Omega^1(Z)$. Note that in general, $\alpha$ is multiplied by a non-vanishing function $g$ that is not vertically invariant. As $g$ is non-vanishing, the form divided by the function gives a contact form defining the same contact structure. The non-integrability condition then is equivalent of saying that {$u(d\beta)^n+n\beta\wedge du \wedge d\beta$} is a volume form on $Z$. We will change the contact form to a $b^{2k}$-contact form.
		
		Take $\epsilon>0$. Let us take a function $s_\epsilon$ (that is smooth outside of $x=0$) such that
		\begin{enumerate}
			\item $s_\epsilon(x)=x$ for $x\in \R\setminus [-2\epsilon,2\epsilon]$,
			\item$s_\epsilon(x)=-\frac{1}{x^{2k-1}}$ for $x\in  [-\epsilon,0[ \cup ]0,\epsilon]$,
			\item $s_\epsilon'(x)>0$ for all $x \in \R$.
		\end{enumerate}
		Now consider $\alpha_\epsilon=uds_\epsilon+\beta$. By construction, $\alpha_\epsilon$ is a $b^{2k}$-form that coincides with $\alpha$ outside of $Z \times (\R \setminus [-2\epsilon,2\epsilon])$. Furthermore, $\alpha_\epsilon$ satisfies the non-integrability condition on $Z \times ]-2\epsilon,2\epsilon[$ because $s'_\epsilon >0$.
		
		The rest of the statement follows from the discussion of Section \ref{structurecriticalset}.
	\end{proof}

	\begin{remark}
		Note that there are many different choices for the function $s_\epsilon$ yielding the same result: the function $s_\epsilon$ only needs to allow singularities of the right order and have positive derivative. Given a convex contact form $\alpha$, we call $(M,\alpha_\epsilon)$ the $s_\epsilon$-singularization.
	\end{remark}

	This proof only works for $b^m$-contact structures where $m$ is even because it is essential that $s_\epsilon'>0$.
	In the case where the complimentary set of the convex hypersurface is connected, the contact condition obstructs the existence of $b^{2k+1}$-contact structures on $M$ having $Z$ as critical set. This is because the contact condition induces an orientation on the manifold, whereas in the $b^{2k+1}$-contact case, the orientation changes when crossing the critical set. The same holds for symplectic surfaces: see for example \cite{MP} where this orientability issues were formulated using colorable graphs.
	
	\begin{lemma}\label{lem:orientability}
	Let $M$ be an orientable manifold with $Z$ a hypersurface such that $M \setminus Z$ is connected. Then there exist no  $b^{2k+1}$-contact form with critical set $Z$.
	\end{lemma}

\begin{proof}
	Assume by contradiction that there is a $b^{2k+1}$-contact form. Let $z$ be a defining function for the critical set. The contact condition writes down as $\frac{dz}{z^{2k+1}} \nu$, where $\nu$ is volume form on $M$. This expression has opposite signs on either side of $Z$. As $M\setminus Z$ is connected, $\alpha \wedge (d\alpha)^n$ must vanish in $M \setminus Z$, which is in contradiction with the contact condition.
\end{proof}

	To overcome this orientability issue, we prove existence of $b^{2k+1}$-contact structures with two disjoint critical sets contained in a tubular neighbourhood of a given convex hypersurface.
	
	\begin{theorem}\label{Existencetheoremb2k+1}
		Let $(M,\xi)$ be a contact manifold and let $Z$ be a convex hypersurface in $M$. Then $M$ admits a $b^{2k+1}$-contact structure for all $k$ that has two diffeomorphic connected components $Z_1$ and $Z_2$ as critical set. The codimension $2$ submanifold $\Sigma$ corresponds to the set where the rank of the distribution drops and the induced structures is contact. Additionally, one of the hypersurfaces can be chosen to be $Z$.
	\end{theorem}
	
	\begin{proof}
		The proof follows from the same considerations as before, but replacing the vertically invariant contact form $\alpha$ defining the contact $\xi$ by $\alpha_\epsilon= uds_\epsilon+\beta$, where $s_\epsilon: ]-\epsilon,\epsilon[ \to \R$ is given by
		\begin{itemize}
			\item $s_\epsilon(t)= |t|$ for $|t| \in [3\epsilon/4,\epsilon]$,
			\item $s_\epsilon(t)= \log|t-3\epsilon/8|$ for $|t| \in [\epsilon/4,\epsilon/2]$ if $m=1$,
			\item $s_\epsilon(t)= \frac{1}{2k(x-3\epsilon/8)^{2k}}$ for $|t| \in [\epsilon/4,\epsilon/2]$ if $m=2k+1 \neq 1$,
			\item $s_\epsilon$ is odd, i.e. $s_\epsilon(-t)= -s_\epsilon(t)$,
			\item $s_\epsilon'(t) \neq 0$.
		\end{itemize}
	As before, $s_\epsilon' \neq 0$ assures that $\alpha_\epsilon$ is a $b^{2k+1}$-contact form. As any other function with non-vanishing derivative and the right order of singularities gives rise to a $b^{2k+1}$-contact form, one of the two hypersurfaces can be chosen to be the initial convex hypersurface.
	\end{proof}

	\begin{remark}\label{rmk2criticalsets}
		Given a contact manifold with a convex hypersurface such that the complementary set of the hypersurface is not connected, it may, in some particular cases, also be possible to construct a $b^{2k+1}$-contact form admitting a unique connected component as critical set. This is related to extending a given contact form in a neighbourhood of a contact manifold with boundary to a globally defined contact form. More precisely, let $\alpha$ be the contact form. In a tubular neighbourhood around the convex hypersurface, we replace as before $\alpha=udt+\beta$ by $\alpha_\epsilon= uds_\epsilon+\beta$  where $s_\epsilon$ is given by
		\begin{itemize}
			\item $s_\epsilon(t)=t$ for $t>2\epsilon$,
			\item $s_\epsilon(t)=\log t$ for $0<t<\epsilon$,
			\item $s_\epsilon'(t)>0$ for $t>0$,
			\item $s_\epsilon$ is even, i.e. $s_\epsilon(-t)=s_\epsilon(t)$.
		\end{itemize}
		The form $\alpha_\epsilon$ is a $b^{2k+1}$-contact form that agrees with $\alpha$ for $t>2\epsilon$. However, it does not agree with $\alpha$ for $t<-2\epsilon$ and in fact, it may not always be possible to extend $\alpha_\epsilon$. \textcolor{black}{This is related to the following question: Given a $3$-manifold and a hypersurface $Z$ defined by a globally defined defining function $f$ (i.e. different from the on in Lemma \ref{lem:orientability}), is it always possible to construct a $b^{2k+1}$-contact structure realizing $Z$ as a critical set?}
	\end{remark}

		\begin{remark}\label{rmk:convexbcontact}
	Recall that a $b^m$-contact form $\alpha=u\frac{dz}{z^m}+\beta$ is convex if $u\in C^\infty(Z)$ and $\beta \in \Omega^1(Z)$, see Definition \ref{def:bcontactconvex}. Note that the $s_\epsilon$-singularization of is by construction a convex $b^m$-contact manifold.
	\end{remark}
	
	As a corollary of Theorem \ref{Existencetheoremb2k+1}, we will prove that any \textcolor{black}{coorientable} contact manifold is folded contact.

	\begin{corollary}\label{cor:existencefoldedcontact}
	Let $Z$ be a convex hypersurface in a coorientable contact manifold $(M,\ker \alpha)$. Then $M$ admits a folded-contact form that has two connected components $Z_1$ and $Z_2$ as folded hypersurface, diffeomorphic to $Z$. In particular in dimension $3$, a surface $Z\subset M$ can be realized as a connected component of the folded hypersurface of a folded contact form.
	\end{corollary}
	
	\begin{proof}
	The proof is a direct application of the existence theorem for $k=1$ and the desingularization theorem. First, by Theorem \ref{Existencetheoremb2k+1}, the hypersurface $Z$ can be realized as one of two connected components of the critical set of a $b$-contact structure. As the obtained $b$-contact form is convex, see Remark \ref{rmk:convexbcontact}, we then use the desingularization theorem (Theorem \ref{thm:desingularizationcontactb2k+1}) to obtain a folded contact structure.
	
	The statement in dimension $3$ follows from Corollary \ref{cor:3dfixed}.
	\end{proof}
	
	\begin{remark}
	It follows from \cite{bem} that a necessary and sufficient condition for a manifold to admit a contact structure is that it is almost contact. It would be interesting to ask whether the almost contact condition can be relaxed to prove the existence of $b^m$-contact structures on closed manifolds. For an example it is well known that $SU(3)/SO(3)$ does not admit a contact structure, see \cite{Geiges}. Another indication for this is given by examples of cooriented $b^m$-contact structures on non-orientable manifolds (see Example \ref{ex:nonorientable}).
	\end{remark}

\section{Open problems}

	We finish this article by mentioning open problems concerning $b^m$-contact manifolds that will be the content of an upcoming paper. In this article, we did not consider the study of the dynamical properties of the Reeb vector field associated to a $b^m$-contact form. A natural question that arises is whether or not its flow always admits periodic orbits. In the case of a positive answer, do those orbit always exist away from the critical set or on the critical set? Those are generalizations of the well-known Weinstein conjecture in contact geometry and is discussed in \cite{MO2}.
	
	This is particularly interesting bearing in mind applications to celestial mechanics. As was mentioned in Section \ref{sec:motivating}, it has been proved that $b^m$-symplectic structures appear naturally in the study of celestial mechanics as for example the restricted three body problem, see \cite{BDMOP,dkm,kms}. As $b^m$-contact manifolds appear as \textcolor{black}{particular cases} of level-sets of Hamiltonian functions and the Reeb flow is a reparametrization of the Hamiltonian flow on this level-set, results in the direction of the before-mentioned generalization of Weinstein conjecture are of great importance in the study of celestial mechanics. These lines of research are carried out in \cite{MO2} and may put some new light in studying the contact geometry of the restricted three body problem, in a similar vein as in \cite{contactmoon}. The dynamical properties of the $b$-Reeb field is narrowly related to the ones of Beltrami vector fields on manifolds with boundary as it is shown in \cite{MOP}.

	\appendix

	\section{Local model of Jacobi manifolds}\label{app:localJacobi}
	
	We recall local structure theorems of Jacobi manifolds, proved in \cite{DLM}. Let us first introduce some notation. Consider $\mathbb{R}^{2q+1}$ with coordinates $(x_0,\dots,x_{2q})$. Following \cite{DLM}, we put
	
	\begin{itemize}
		\item $\Lambda_{2q}= \sum_{i=1}^q \frac{\partial}{\partial x_{i+q}}\wedge \frac{\partial}{\partial x_i}$,
		\item $Z_{2q}= \sum_{i=1}^q x_{i+q}\frac{\partial}{\partial x_{i+q}}$,
		\item $R_{2q+1}= \frac{\partial}{\partial x_0}$,
		\item $\Lambda_{2q+1}= \sum_{i=1}^q(x_{i+q}\frac{\partial}{\partial x_0}-\frac{\partial}{\partial x_i})\wedge \frac{\partial}{\partial x_{i+q}}$.
	\end{itemize}
	\begin{theorem}[\cite{DLM}]
		Let $(M^m,\Lambda,R)$ be a Jacobi manifold, $x_0$ a point of $M$ and $S$ be the leaf of the characteristic foliation going through $x_0$.
		
		If $S$ is of dimension $2q$, then there exist a neighbourhood of $x_0$ that is diffeomorphic to $U_{2q} \times N$ where $U_{2q}$ is an open neighbourhood containing the origin of $\R^{2q}$ and $(N,\Lambda_N,R_N)$ is a Jacobi manifold of dimension $m-2q$. The diffeomorphism preserves the Jacobi structure, where the Jacobi structure on $U_{2q} \times N$ is given by
		$$R_{U_{2q} \times N}= \Lambda_N, \qquad R_{U_{2q} \times N}= \Lambda_{2q}+\Lambda_N-Z_{2q}\wedge R_N.$$
		
		If $S$ is of dimension $2q+1$, then there exist a neighbourhood of $x_0$ that is diffeomorphic to $U_{2q+1} \times N$ where $U_{2q+1}$ is an open neighbourhood containing the origin of $\R^{2q+1}$ and $(N,\Lambda_N,R_N)$ is a homogeneous Poisson manifold of dimension $m-2q-1$. The diffeomorphism preserves the Jacobi structure, where the Jacobi structure on $U_{2q} \times N$ is given by
		$$R_{U_{2q+1} \times N}= R_{2q+1}, \qquad \Lambda_{U_{2q+1} \times N}= \Lambda_{2q+1}+\Lambda_N+E_{2q+1}\wedge Z_N.$$
	\end{theorem}

\end{document}